\documentclass[10pt,a4paper]{article}
\usepackage[utf8]{inputenc}
\usepackage[english]{babel}
\usepackage{amsmath}
\usepackage{amsthm}
\usepackage{float}
\usepackage{amsfonts}
\usepackage{amssymb}
\usepackage{graphicx}
\usepackage{mathrsfs}
\usepackage{enumerate}
\usepackage{mathtools} % for \DeclarePairedDelimiter
\usepackage{xfrac} % for nice 1/2, 1/4 etc with \sfrac{1}{2}
\usepackage{geometry}
\usepackage{appendix}
\usepackage{authblk}

\makeatletter
\newtheorem*{rep@theorem}{\rep@title}
\newcommand{\newreptheorem}[2]{%
\newenvironment{rep#1}[1]{%
 \def\rep@title{#2 \ref{##1}}%
 \begin{rep@theorem}}%
 {\end{rep@theorem}}}
\makeatother

\newreptheorem{lemma}{Lemma}
\newreptheorem{proposition}{Proposition}
\newreptheorem{theorem}{Theorem}
\newreptheorem{corollary}{Corollary}
\usepackage{tikz}
\usepackage[normalem]{ulem} % for \sout{}, strikeout
\usepackage[all,cmtip]{xy}

\usepackage[autostyle]{csquotes} % recommended by biblatex
\usepackage[
  backend=biber,
  style=alphabetic,
  giveninits=true, % Use initials instead of full name
  terseinits=true, % Instead of E. L. K. writes ELK etc.
  maxnames=6, % How many names to show max before truncating to et. al.
  natbib=true, % Allow use of natbib-like commands
  sortlocale=en_GB,
  url=false,
  doi=true,
  eprint=true
  ]{biblatex}
\usepackage{hyperref}
\usepackage{xcolor}
\hypersetup{
  colorlinks=true,
  linkcolor={blue!80!black},
  citecolor=black,
  urlcolor={blue!80!black}
}

\theoremstyle{plain}
\newtheorem{theorem}{Theorem}[section]
\newtheorem{lemma}[theorem]{Lemma}

\newtheorem{proposition}[theorem]{Proposition}

\theoremstyle{definition}
\newtheorem{definition}[theorem]{Definition}

\theoremstyle{remark}
\newtheorem{remark}[theorem]{Remark}
\newtheorem{example}[theorem]{Example}

\numberwithin{equation}{section}

\newcommand{\R}{\mathbb{R}}
\newcommand{\C}{\mathbb{C}}
\newcommand{\N}{\mathbb{N}}

\newcommand{\E}{\mathbb{E}} % expected value
\DeclarePairedDelimiter\parentheses{\lparen}{\rparen} % requires mathtools packages
\DeclarePairedDelimiter\abs{\lvert}{\rvert} % use \abs* for auto-sized bars

\newcommand{\supp}{\operatorname{supp}}
\newcommand{\intlim}[2]{\bigg/_{\!\!\!\! #1}^{\, #2}}
\newcommand{\ds}{\,\mathrm{d}s}
\newcommand{\dt}{\,\mathrm{d}t}
\newcommand{\dx}{\,\mathrm{d}x}

\newcommand{\Cyl}{\text{Cyl}}

\newcommand{\DN}{\Lambda_{\text{DN}}}
\newcommand{\ND}{\Lambda_{\text{ND}}}

\title{Recovering a $(1+1)$-dimensional wave equation from a single white noise boundary measurement}
\author[,1]{Emilia L.K. Bl{\r a}sten\thanks{\href{mailto:emilia.blasten@lut.fi}{emilia.blasten@lut.fi}}}
\author[,1]{Tapio Helin\thanks{\href{mailto:tapio.helin@lut.fi}{tapio.helin@lut.fi}}}
\author[,1]{Antti Kujanpää\thanks{\href{mailto:antti.kujanpaa@lut.fi}{antti.kujanpaa@lut.fi}}}
\author[,2]{Lauri Oksanen\thanks{\href{mailto:lauri.oksanen@helsinki.fi}{lauri.oksanen@helsinki.fi}}}
\author[,1,3]{Jesse Railo\thanks{\href{mailto:jesse.railo@lut.fi}{jesse.railo@lut.fi}; \href{mailto:railo@stanford.edu}{railo@stanford.edu}}}
\affil[1]{Computational Engineering, School of Engineering Sciences,
Lappeenranta-Lahti University of Technology LUT, Lappeenranta, Finland}
\affil[2]{Department of Mathematics and Statistics, University of Helsinki, Helsinki, Finland}
\affil[3]{Department of Mathematics, Stanford University, Stanford, CA, USA}

\begin{document}

\maketitle

\begin{abstract}
  We consider the following inverse problem: Suppose a $(1+1)$-dimensional wave equation on $\R_+$ with zero initial conditions is excited with  Neumann boundary data modelled as a white noise process. Given also the Dirichlet data at the same point, determine the unknown first order coefficient function of the system.
    
  We first establish that {the} direct problem is well-posed.  The inverse problem is then solved by showing that correlations of the boundary data determine the Neumann-to-Dirichlet operator in the sense of distributions, which is known to uniquely identify the coefficient. This approach has applications in acoustic measurements of internal cross-sections of fluid pipes such as pressurised water supply pipes and vocal tract shape determination.
\end{abstract}


%\tableofcontents

 % {TODO after referee comments:
 %   \begin{itemize}
 %   \item Add references to \cite{smirnov20_convex_1d_hyper_coeff_inver,smirnov20_convex_inver_probl_1d_wave,le22_carlem_contr_mappin_1d_inver} somewhere. These were added by Lauri.
 %   \item Mention other references updated since submission
 %   \item We need to change the main conclusion: we get uniqueness for any countable collection almost surely. A measure theoretic thing.
 %   \end{itemize}
 % }

\section{Introduction}
%\footnote{J: Sähköpostit pitää vielä lisätä. Lisätty. tarkistakaa.-A}

In this paper, we deal with an inverse problem for the wave equation in a one-dimensional domain. Ideally, we aim to solve the following problem:

\begin{equation}
  % \tag{...} & hyperref, ks. https://tex.stackexchange.com/a/488373
  \tag{P}\label{quote:ideal-problem}
  \parbox{\dimexpr\linewidth-2cm}{%
    \strut
    Suppose the fluid pressure or flow is measured at a single location in a pipe, with little to no information about the source signals, such as incident or source waves. Can the internal geometry of the pipe be recovered?
    \strut
  }
\end{equation}
% Ks. https://tex.stackexchange.com/a/488373
\renewcommand*{\theHequation}{eitagia.\theequation}

A problem setting such as the one above is typically referred to as passive imaging, since the goal is to image a medium without generating incident or probing waves. Mathematically, (\ref{quote:ideal-problem}) corresponds to recovering a parameter function of interest of a partial differential equation, given its single solution observed at a specific location. The difficulty of the problem depends on the information we have regarding the source signal. In general, this type of problem is severely underdetermined in spatial dimensions two and greater, where classical methods from inverse problems fail and many open problems exist e.g. Schiffer's problem of determining the support of a scattering potential from a single scattering measurement \cite{lax89_scatt_theor}. In contrast, the one-dimensional case is significantly more tractable, as the problem becomes formally determined. Indeed, in cases where the incident probing wave can be chosen freely, the inverse problem for the one-dimensional wave equation has been solved \cite{gelfand55_deter_differ_equat_from_its_spect_funct,sondhi_gopinath}. These results, along {with} many subsequent works such as \cite{bukhgeim1981global, cheng2002identification, cheng2009uniqueness, feizmohammadi2023global, helin2012inverse, helin2014inverse, Rakesh_2001, stefanov2013recovery, helin2020inverse}, use a specific probing wave. 
We highlight, in particular, a recent line of work based {on} the so-called convexification method and its variations \cite{smirnov20_convex_1d_hyper_coeff_inver,smirnov20_convex_inver_probl_1d_wave,le22_carlem_contr_mappin_1d_inver}.

{In contrast, the inverse hyperbolic problem where the incident wave cannot be chosen freely is much less studied. One of the best results so far is \cite{romanov20_recon_princ_coeff_damped_wave}. They not only determine the unknown coefficient based on output from a fixed incident wave with very minimal smoothness restrictions, but also describe a reconstructions scheme to situations where the output has added $L^2$-noise.}

The pure passive imaging setting, where also the source signal is unknown, has been explored in the literature as well. Very recently, Feizmohammadi proposed a solution to the one-dimensional passive imaging problem in the context of the wave equation in \cite{ali_passive}. This work improves upon previous results by removing restrictions on the admissible source and wave speed functions. For recent examples of such results, see \cite{knox2020determining, liu2015determining, kian2023determination, avdonin2014reconstructing, finch2013transmission, jing2021simultaneous}. Another perspective on passive imaging, in two and higher dimensions, is explored in recent work on the inverse problem of corner scattering, where boundary measurements performed with a single arbitrary wave recover, for example, information about the geometry and location of the scatterer \cite{blåsten14_corner_alway_scatt, hu16_shape_ident_inver_medium_scatt, blåsten20_unique_deter_shape_scatt_screen}.

In this paper, we adopt a different approach to passive imaging by analyzing the statistical properties, specifically the temporal correlations, of wave systems that exhibit ambient pressure fluctuations. Such a imaging modality, often termed \emph{correlation imaging}, exploits the available statistical information about the ambient noise in the system and has great potential for imaging remote or hard to access locations \cite{buckingham92_imagin_ocean_with_ambien_noise,gozum22_noise_based_high_resol_time}. Moreover, correlation imaging generally results in less ill-posed problem formulations compared to the underlying deterministic identification problem and even enables partial recovery in some settings where no deterministic identification results are known, see e.g. \cite{devaney1979inverse, garnier2009passive, garnier2012correlation,  bao2017inverse, helin2018atmospheric, caro2019inverse, agaltsov2020global, li2020inverse, li2022inverse1,li2022inverse2, triki2024fourier,  li2024stability}. In particular, we highlight earlier work by the authors in \cite{helin18_correl_based_passiv_imagin_b}, which is closely related to the present paper. Our contribution is to extend correlation imaging to a setting where we have a point measurement, with time as the only variable in our data.
For an extensive review on correlation imaging, we refer the reader to the monograph \cite{GP_passive}.

Let us consider the potential applications of the specific problem setting of determining the internal geometry of a pipe in (\ref{quote:ideal-problem}). An important problem in civil engineering is the maintenance and diagnosis of water supply pipes \cite{gozum22_noise_based_high_resol_time}.  Mineral buildup and other deposits can reduce the internal cross-sectional area, leading to pressure drops and energy losses. However, excavating pipes in urban areas is costly and disruptive. Therefore remote sensing and indirect measurements are favored in problem diagnosis. Previous work by one of the authors on the Smart Urban Water Supply Systems (SmartUWSS) looked at solutions to this problem (\cite{zouari19_inter_pipe_area_recon_as} and others) in various settings based on measuring the \emph{impulse-response function} such as in \cite{sondhi_gopinath}. A challenge in that context is how to produce the incident $\delta$-function for imaging the pipe. Typically, the measurement is performed by some other smooth probing wave (e.g. a smoothened Heaviside function) followed by signal processing to approximatively reconstruct {the} system's response to an ideal $\delta$-function. In contrast, noise from traffic and other sources is ubiquitous in high pressure water supply pipes and correlating the pressure measurements in different locations can determine the internal cross-sectional area of a pipe \cite{gozum22_noise_based_high_resol_time}. For this application perspective, our work explores whether a single-point measurement provides enough information to solve the problem.

A second application is related to human speech. Vowel production can be modelled as a linear system of a source signal at the glottis travelling through the vocal tract \cite[Chapter 6]{titze00_princ_voice_produc}. Voice synthesis, speech recognition and medical applications all benefit from the knowledge of the shape of the vocal tract and source function at the glottis. \emph{Inverse glottal filtering} is a major research question in acoustics \cite{alku11_glott_inver_filter_analy_human}. The paper by Sondhi and Gopinath \cite{sondhi_gopinath} solved the problem of vocal tract shape reconstruction but they use active measurements (a $\delta$-function as an incident wave). The problem (\ref{quote:ideal-problem}) corresponds to simply listening to someone speak a vowel, and to determine their vocal tract configuration based on the audio alone. It is something that humans are quite adept at doing automatically (especially people good at imitating!) but machines and algorithms struggle with.

\smallskip In this paper, we take a step towards solving (\ref{quote:ideal-problem}) by studying what happens when the system is probed with noise and we can measure the response only in one location. { There is previous work dealing with inverse problems involving noise. Many deal with higher {spatial} dimensions and the recovery of a random source or potential such as \cite{bao16_inver_random_sourc_scatt_probl_sever_dimen,liu23_inver_probl_random_schro_equat}. Recent work for the hyperbolic problem determines the unknown coefficient in cases where additive noise pollutes the theoretical measurement \cite{korpela16_regul_strat_inver_probl_1,romanov20_recon_princ_coeff_damped_wave}. The more recent of these works also allows for an arbitrary but regular incident probing wave. In contrast, our problem of interest considers noise to be the inherent probing signal, and thus our methods for solving the problem rely on tools from stochastics for reducing the problem to a deterministic previously solved one.}

Our method requires the recording of the incident white noise and the response from the pipe. These are then correlated over time. We prove that this operation reconstructs the theoretical impulse response function which according to Sondhi and Gopinath \cite{sondhi_gopinath} determines the internal cross-sectional area function of the pipe. Mathematically this means that the {Dirichlet} and Neumann data of a single realisation of a white noise solution to a (1+1)-dimensional wave equation {determine} the unknown first order coefficient.

% \subsection{The main theorem}
\medskip We will describe our result, including the notation next. In models, the {following} function $A$ represents the internal cross-sectional area at points along the axis of the pipe.
\begin{definition}\label{def:admissible}
%{(I wrote this for smooth $A$ but I think it works for $A \in H^2_{loc}$ too. I couldnt find a decent reference for Cauchy problem or initial boundary value problem with non-smooth coefficients. I think the theory exists though. Oho, anteeksi kirjoitin englanniksi. 

%Huomatkaa myös, että tavallinen (oppikirjamainen) initial-boundary value ongelmien teoria ei toimi meidän tapauksessa, koska white noise ei ole tarpeeksi regulaarinen. Siksi tuo suora ongelma on niin rasittavan pitkä preliminary kappaleessa. Mä oon enemmän kuin valmis poistamaan nuo suoran ongelman laskut preliminary kappaleesta pois, jos löytyy joku hyvä lähde, millä ne voi korvata --Antti)} 
A function $A : \R \to \R$  is called admissible if there are  points $x_+> x_- > 0$ and a constant $A_\infty >0$ such that 
\begin{align}
&A \in C^\infty(\R), \\
& A > 0, \\
& A (x) = 1 , \quad \text{for every} \quad x < x_- , \label{thirdco} \\
&A (x) = A_\infty, , \quad \text{for every } \quad x > x_+. 
\end{align}
\end{definition}  
\begin{remark}
The 
%\footnote{J: Vähän hämäävä, kun tuossa on tuo numerointi, joka eroaa.} 
condition \eqref{thirdco} could be replaced with $A (x) = A_{-\infty}$, $\forall  x < x_-$, where $A_{-\infty} >0 $ is a constant. One can always choose units such that $ A_{-\infty} = 1$, however. 
\end{remark}

Let  $t \mapsto \chi(t)$ be a smooth cut-off function that equals $1$ for large $t$ and is supported in $(0,\infty)$. 
That is; $\chi \in C^\infty( \R) $, $\text{supp} (\chi) \subset (0,\infty)$ and  $\chi-1 \in C^\infty_c(0,\infty)$. 
Let $\omega \in \mathcal{S}'(\R)$ be an arbitrary element in the Gaussian white noise space $(\mathcal{S}'(\R), \mathcal{F}, \mathbb{P})$,
%{(Valitsin tähän Borel mutta cylindrical $\sigma$-algebra käynee myös)} 
see Section \ref{Gaussi} for more details. 
The product $\chi \omega$ then represents a white noise signal that is switched on smoothly after $t=0$. 
Considering $\chi \omega$ as the Neumann boundary value, 
the wave $u$ in $(0,\infty) \times \R$ generated by the signal is modeled by
%\footnote{J: Pitäisikö tässä jotenkin korostaa, että tämä on määritelmä tai malli, joka valitaan. Esim. ", {we consider} the wave $u$ [...] solving the equation [equation lines]"? }
\begin{align}
&\square_A u (x,t) = 0 \quad \text{for} \quad (x,t) \in  (0,\infty) \times  \R \label{th-ibvp1}  \\
&u(x,t)  = 0 ,  \quad \text{for}  \quad (x,t) \in  (0,\infty)\times (-\infty,0) \label{th-ibvp2} \\
&\partial_x u (x,t) \big|_{x  = 0}  = \chi(t)\omega(t) , \quad \text{for}  \quad t \in \R,  \label{th-ibvp4}
\end{align}
where $\square_A f = \partial_t^2 f - \partial_x^2 f - (\partial_x A / A) \partial_x f$. { This form of the wave operator is motivated by the study of the water hammer effect in civil engineering \cite{wylie93_fluid_trans_system,ghidaoui05_review_water_hammer_theor_pract}. However we point out that coordinate changes and gauge transforms can be used to move the unknown coefficient to the second or zeroth order terms, making this formulation equivalent with other forms of the hyperbolic second order equation.} Let $ \ND $ stand for the Neumann-to-Dirichlet map (see Definition~\ref{DN-def}) which takes the Neumann boundary value $\chi \omega$ to the corresponding Dirichlet value $u|_{x=0}$. We may also write $\Lambda := \ND$. 
%The research question is the following:
%Is it possible to tell the difference between  two profiles $A$ from a single boundary measurement $\Lambda (\chi \omega)$, where $\omega \in \mathcal{S}'(\R)$ is an arbitrary realization of white noise?
The following theorem is the main outcome of this article:
%\footnote{J: Voidaanko tässä kohtaa esitellä konventio, että kirjoitetaan $\Lambda = \Lambda_{ND}$, mikäli ei ole sekaannuksen vaaraa?}

\begin{theorem}\label{themain}
Let the cut-off function $\chi \in C^\infty( \R)$ be supported in $(0,\infty)$ and satisfy $\chi-1 \in C^\infty_c(0,\infty)$. 
Assume that $A_\nu \in C^\infty(\R)$, $\nu = 1,2$ are admissible and 
denote by ${\Lambda}_{\nu}$ the Neumann-to-Dirichlet map (Definition \ref{DN-def}) corresponding to $ A_\nu$. 
Then, $A_1 \neq A_2$ if and only if there is a measurable set $U \subset \mathcal{S}'(\R)$ with $\mathbb{P}(U) = 1$  such that 
$\Lambda_{1} ( \chi \omega) \neq \Lambda_{2} ( \chi \omega)$ for every $ \omega \in U$. 
That is; any single measurement separates the functions $A_\nu$ almost surely.
%\footnote{J: Onko $\lambda$ tässä kontekstissa tyypillinen valinta indeksiksi? Voisiko joku perinteisempi valinta toimia (esim. $i$, $j$ tai $k$)? A: en tohdi vaihtaa noihin, koska ne on varmasti jo käytössä muualla. Laitoin nu} 
\end{theorem}

{ Before going to the proofs, let us compare this result to previous work. Very classical results, as considered in \cite{gelfand55_deter_differ_equat_from_its_spect_funct,sondhi_gopinath} require a very specific input function such as a full spectrum or delta distribution in place of our $\chi \omega$. On the output side there is more similarity: both require the full knowledge of the theoretical measurement. \cite{korpela16_regul_strat_inver_probl_1} requires $\delta_0(t)$ as their input, but has a regularization scheme to deal with additive noise in the output measurement. In this consideration, perhaps the closest to our result is \cite{romanov20_recon_princ_coeff_damped_wave}. Like the previous work, they deal with additive noise in the output. However like in the theorem above, they allow for a non-specific input, in fact any non-trivial $C^2$-function vanishing for $t<0$. Integrating our input $\chi \omega$ suitably many times with respect to time would produce a $C^2$-function whose output in our system is the same number of integrals of $\Lambda(\chi \omega)$. Then one could use the methods introduced in \cite{romanov20_recon_princ_coeff_damped_wave} to recover the unknown coefficient. In this paper we consider a different type of proof: we use a correlation operator \eqref{eq:correlation-function-def} and a recent energy estimate \cite{arnold22_expon_time_decay_one_dimen}. This recovers the full Neumann-to-Dirichlet operator in the sense of distributions and the conclusion follows.}

The following diagram represents the structure of the proof:
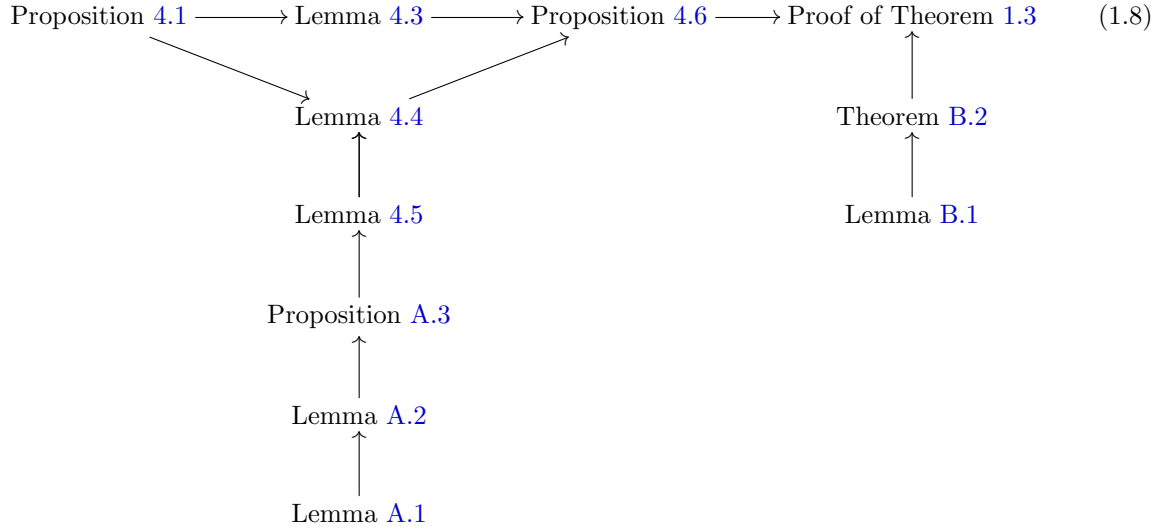
\begin{equation}\label{tree1}
\xymatrix{
\text{Proposition \ref{prop:integration-by-parts}}  \ar[r] \ar[dr] & \text{Lemma \ref{weaklemma}} \ar[r] & \text{Proposition \ref{L2-con}}  \ar[r] & \text{Proof of Theorem \ref{themain}} \\
    & \text{Lemma \ref{E-lemma}}  \ar[ur] & &  \text{Theorem \ref{sondhi_theorem}} \ar[u]  \\
    & \ar[u] \text{Lemma \ref{K-lemma}} \ar[u]& & \text{Lemma \ref{sondhi_lemma}} \ar[u]  \\
    & \ar[u] \text{Proposition \ref{1decay-pro}} & \\
 &  \ar[u] \text{Lemma \ref{lem:measurement-decay}}& \\
 &  \ar[u]  \text{Lemma \ref{gronw-lemma}} &
    }
\end{equation}

Theorem \ref{themain} is derived in Section \ref{derivation-section}. 
The key element is Proposition \ref{1decay-pro} which gives a relation between random white noise measurements and the standard Neumann-to-Dirichlet map.  
This result relies on energy estimates studied separately in Section \ref{energy-estimate-sec}. 
The actual proof of Theorem \ref{themain} is a combination of Proposition \ref{L2-con} and  the techniques  developed in  \cite{sondhi_gopinath}. 
%See Appendix \ref{sondhi_app} for the latter. 

The overall structure of the article is as follows: 
In Section \ref{preliminaries-sec}, we outline the basic mathematical concepts used in this work. After that, the ``direct problem'' is studied in Section \ref{IBVP}. The main theorem is then  proved in Section  \ref{derivation-section} with the use of estimates derived in Section \ref{energy-estimate-sec}. In the appendix,  some results from  \cite{sondhi_gopinath} are put in the context of this article.  

%\textcolor{red}{Proof summary\ldots}

%{Notation\ldots
%sisällytä $\square_A u (x,t)= 0$ ja $\square_A u (x,t) :=   \partial_t^2 u (x,t) - \frac{1}{A(x)} \partial_x ( A(x) \partial_x u (x,t) ) $
%}

%\begin{remark}
%asdasd
%\end{remark}

\subsection*{Acknowledgements}
{ We would like to thank the anonymous referees whose patient reading and comments have improved our manuscript and knowledge of previous work in the many fields in whose intersection this paper is located.}

This work was supported by the Research Council of Finland through the Flagship of Advanced Mathematics for Sensing, Imaging and Modelling (decision numbers 359182 and 359183), Centre of Excellence of Inverse Modelling and Imaging (decision numbers 353094 and 353096), the Finnish Ministry of Education and Culture’s Pilot for Doctoral Programmes (Pilot project Mathematics of Sensing, Imaging and Modelling) and project grants (decision numbers 347715 and 348504). In addition, L.O. was supported by the European Research Council of the European Union, grant 101086697 (LoCal). In addition, J.R. was supported by Emil Aaltonen Foundation, Fulbright Finland Foundation (ASLA-Fulbright Research Grant for Junior Scholars 2024--2025), and Jenny and Antti Wihuri Foundation. Views and opinions expressed are those of the authors only and do not necessarily reflect those of the European Union or the other funding organizations.

\section{Preliminaries}\label{preliminaries-sec}

\subsection{Test functions and distributions}\label{subsection_test-functions}

Let $C^\infty(\R^n)$ stand for the space of smooth functions on $\R^n$. 
We denote by $C^\infty_c(\R^n)$ the space of compactly supported elements in $C^\infty(\R^n)$. 
For a compact $K \subset \R^n$,  let $C^\infty_K (\R^n)$ stand for functions $\phi \in C^\infty_c (\R^n)$ supported in $K$. 
Each $C_{K}^\infty( \R^n)$ admits a Fr\'echet space topology  induced by the seminorms 
\[
\| \phi \|_{k,K} := \sum_{|\alpha| \leq k} \sup_{x\in K}   | \partial^\alpha \phi(x) |, \quad k =0,1,2,3\dots
\]

For any nested sequence $K_1 \subset K_2 \subset K_3 \subset \cdots$ of compact sets (e.g. $K_j = \bar{B}(0,j)$)  such that $\bigcup_{j\in \N} K_j = \R^n$ we have $C_c^\infty ( \R^n) = \bigcup_{j \in \N} C_{K_j}^\infty$. 
%Moreover, the topology  in $C_{K_j}^\infty( \R^n)$ 
%coincides with the subspace topology (also known as the  relative topology or induced topology) inherited from  $C_{K_{j+m}}^\infty( \R^n)$, where $m \in \N$.  
We equip the space $C_c^\infty ( \R^n ) $ 
with the inductive limit topology, that is, the largest topology for which the trivial inclusion maps $ C_K^\infty  ( \R^n ) \hookrightarrow  C_c^\infty ( \R^n ) $ are continuous. 
It makes $C_c^\infty ( \R^n ) $ into a topological vector space. 
It is straightforward to check that the inductive limit topology does not depend on the choice of the nested compact sets $K_j$ covering $\R^n$. The elements of $C_c^\infty ( \R^n ) $ are called smooth (compactly supported) test functions.

Let $\mathcal{D}'(\R^n)$ be the dual space of $C^\infty_c (\R^n)$. 
By definition, an element in $ \mathcal{D}'(\R^n)$ is a continuous linear map $u : C^\infty_c (\R^n) \to \C $. Equivalently, $u |_{C^\infty_K (\R^n)}$  is a continuous linear map $C^\infty_K (\R^n) \to \C $ for every compact $K \subset \R^n$. 
It follows that a linear map 
%$u \in \mathcal{D}'(\R^n)$
%if and only if 
$u : C^\infty_c ( \R^n) \to \C $ lies in $\mathcal{D}'(\R^n)$  if and only if 
for every compact $K \subset \R^n$ there is a constant $C\geq 0$
and $k \in \N$ such that 
\begin{equation}\label{distiey}
| \langle u, \phi \rangle | \leq C \| \phi \|_{k,K}  
%C  \sum_{|\alpha| \leq k} \sup_{x \in K}   | \partial^\alpha \phi |
\end{equation}
for every $\phi \in C_K^\infty(\R^n)$.
This is often taken as a definition for distributions. The space of compactly supported distributions in $\R^n$ shall be denoted by $\mathcal{E}'(\R^n)$. This is the dual of $C^\infty (\R^n)$ and can be identified as a subspace of $\mathcal{D}'(\R^n)$.
%
%The topology in $\mathcal{E}'(\R^n)$ is the subspace topology induced from $\mathcal{D}'(\R^n)$. 
%This is the dual of $C^\infty(\R^n)$. We have  
%\[
%C_c^\infty(\R^n) \subset \mathcal{S}(\R^n) \subset  %C^\infty(\R^n), 
%\]
%and 
%\[
%\mathcal{E}'(\R^n) \subset \mathcal{S}'(\R^n) \subset %\mathcal{D}'(\R^n).
%\]
For more details on test functions and distributions, see e.g. \cite{friedlander-joshi-book}. The inductive limit topology is treated in the appendix of the book.  

\subsection{The wave front set of a distribution}

%The Fourier transform $\widehat\phi \in C^\infty (\R^n)$ of a test function $\phi \in C^\infty_c( \R^n)$ is defined by 
%\[
%\widehat\phi (\xi)  := \frac{1}{(2\pi)^{n/2}} \int e^{-%ix \xi}  \phi (x) dx.
%\]
%More generally, the Fourier transform $\widehat{u}$ of %a distribution $u\in \mathcal{E}'(\R^n)$ is defined by 
%\[
%\langle \widehat{u}, \phi \rangle = \langle u, %\widehat\phi \rangle, \quad \phi \in C^\infty_c(\R^n).
%\]
%Here $\chi \in C_c^\infty( \R^n) $ is any smooth test function that equals 1 on the support of $u$.  
Denote $\dot\R^n := \R^n \setminus \{0\}. $ 
By Paley-Wiener theorem, a distribution $u \in \mathcal{D}'(\R^n)$ is smooth in a neighbourhood of $x \in \R^n$ if and only if there is $\phi \in C^\infty_c (\R^n)$ such that $\phi(x) \neq 0$ and 
\begin{equation}\label{palew}
\widehat{\phi u} (\xi )  = O( |\xi|^{-N} ), 
\end{equation}
%uniformly 
for every $\xi \in \dot\R^n $ and $N \in \N$  as $|\xi| \to \infty$. Here $\widehat{\phi u}$ stands for the Fourier transform of $\phi u \in \mathcal{E}'(\R^n)$. That is;
\[
\widehat{\phi u}( \xi) = \langle e^{-i(\cdot,\xi)} u ,  \phi \rangle . 
\]
The wave front set $WF(u) \subset \R^{n} \times \dot\R^n $ of $u \in \mathcal{D}'(\R^n)$ can be defined as the complement of all $(x_0, \xi_0) \in     \R^{n} \times \dot\R^n  $ for which  there exists an open neighbourhood $U $ of $x_0$ in $\R^n$, an open conic neighbourhood $V$ of $\xi_0$ in $\dot\R^n$ and a test function $\phi \in C_c^\infty(U)$, $\phi (x_0) \neq 0$ such that  for every $N \in \N$ there is a constant $C\geq 0$ satisfying the estimate 
\[
| \widehat{\phi u} ( \xi ) | \leq  C \langle \xi \rangle^{-N} . 
\]
in $V$. 
Here $\langle \xi \rangle := ( 1 + |\xi|^2)^{1/2} $. 
 %a constant $C$ possibly depending on $N$. 

The image of $WF(u)$ in the projection $(x,\theta) \mapsto x$ is called the singular support of $u$. It is denoted by $\text{singsupp}(u)$. The singular support of $u$ is the collection of points in $\R^n$ at which $u$ fails to be smooth. In microlocal analysis, such points are called singularities. 
For an introduction to the topic, see e.g. \cite{duistermaat-FIOs,grigis-sjostrand}.

\begin{example}
Let $\delta_0 \in \mathcal{D}'(\R^n) $ be the Dirac delta, defined by 
\[
\langle \delta_0 , \phi \rangle := \phi(0), \quad  \phi \in C^\infty_c( \R^n) . 
\]
The wave front set and singular support of $\delta_0$ are 
\[
WF ( \delta_0 ) = \{0\} \times \dot\R^n  
\quad \quad 
\text{and} 
\quad \quad 
\text{singsupp}(\delta_0) = \{0\}, 
\]
respectively. 
\end{example}
\begin{example}
Let $ u  \in \mathcal{D}'(\R^2) $, be the characteristic function of the half-space $\R_+ \times \R$. As a distribution,  
\[
\langle u , \phi \rangle =  \int_{-\infty}^\infty  \int_{0}^\infty \phi(x,y) dxdy, \quad \phi \in C_c^\infty( \R^2) . 
\]
The wave front set and singular support of $u$ are 
\[
WF(u) = \{0\} \times \R \times  \dot\R \times \{0\}  
\]
and 
\[
\text{singsupp}(u) = \{0\} \times \R, 
\]
respectively.  
\end{example}

\subsection{The space \texorpdfstring{$\mathcal{D}_\Sigma'(\R^n)$}{D’\_Σ(Rn )}}
Here we follow \cite{gabor-wf}. See also \cite[Definition 1.3.2]{duistermaat-FIOs}. 
Let $\Sigma $ be a closed conic set in the punctured cotangent bundle $(T^* \R^n) \setminus \{0\} = \R^n \times \dot\R^n$. 
Later we shall set $n=2$ and $\Sigma= \{ (x,t,\xi_t,\xi_x) \in \R^2 \times \dot\R^2 : \xi_t^2 = \xi_x^2 \}.   $
Define 
\[
\mathcal{D}_{\Sigma}'(\R^n) := \{ u \in \mathcal{D}'( \R^n ) :  WF(u)  \subset  \Sigma \} .
\] 
%Recall that the estimates \eqref{distiey} 
%give an equivalent definition for a linear $u: C^\infty_c(\R^n) \to \C$ to belong to the 
%space $\mathcal{D}'(\R^n)$. 
Recall that the rapid decay estimate \eqref{palew} holds away from $\Sigma$ for every $u \in \mathcal{D}_{\Sigma}'(\R^n)$. Hence, we may consider the seminorms 
\begin{equation}\label{gaborseminorm}
  \rho_{N,\Gamma,\phi}( u )  :=  \sup_{\xi \in \Gamma} | \xi |^N  | \widehat{u \phi }( \xi ) |
\end{equation}
where $N \in \N$, $\phi \in C_c^\infty( \R^n) $ and $\Gamma \subset \R^n  $ is a closed conic set such that $\text{supp} (\phi) \times \Gamma  $ is disjoint of $\Sigma$. 
We equip $\mathcal{D}_\Sigma'(\R^n)$  with the topology generated by this family of seminorms  and the subspace topology inherited from   $\mathcal{D}'(\R^n)$. 
A sequence  $u_j$, $j=1,2,3,\dots$ of elements in $\mathcal{D}_\Sigma'(\R^n)$ 
converges to an element $u$  in $\mathcal{D}_\Sigma'(\R^n)$ if and only if it converges to $u$ in $\mathcal{D}'(\R^n)$  (i.e. $\langle u_j , \varphi \rangle \to \langle u , \varphi \rangle $ for every $\varphi \in C_c^\infty( \R^n) $)  and $ \rho_{N,\Gamma,\phi}( u-u_j ) \to 0$ for all $N\in \N$, $\phi \in C_c^\infty( \R^n) $ and closed conic sets $\Gamma \subset \R^n  $ with $\text{supp}(\phi) \times \Gamma $ disjoint of $\Sigma$.

%asdasddasd

\subsection{Gaussian white noise}\label{Gaussi}

Let $\mathcal{S}(\R^n)$  be the Schwartz space of rapidly decreasing smooth functions on $\R^n$, $n \in \N$.
It is a {Fr\'echet} space under the family of seminorms 
\[
\| u \|_{\alpha,\beta} := \sup_{x \in \R^n }  | x^\alpha \partial_\beta u(x)|, \quad \alpha \in {\N^n}, \quad  \beta \in {\N^n}. 
\]
Let $\mathcal{S}'(\R^n)$ stand for the dual space of $\mathcal{S}(\R^n)$. 
The space $\mathcal{S}'(\R^n)$ is known as the space of tempered distributions. See  \cite{friedlander-joshi-book} for more details. 

%\footnote{J: Laitetaanko joku yleisviite, esim. Hörmander, Triebel tai vastaava?}
%
We equip $\mathcal{S}'(\R^n)$ with 
 the cylindrical $\sigma$-algebra $\mathcal{F}:=\Cyl ( \mathcal{S}'(\R^n))$.
%\footnote{A: Muutin tämän sylinterialgebraksi. Myös Borel käy. }
 %\mathcal{B} ( \mathcal{S}'(\R^d))$ {(the %cylindrical $\sigma$-algebra $\Cyl ( %\mathcal{S}'(\R^d))$  käy myös. The cylindrical %$\sigma$-algebra is a subset of the Borel $\sigma$-%algebra. )}. 
The Bochner-Minlos theorem \cite[Theorem 2.1]{hida-lectures-book}
 implies the existence and uniqueness of Gaussian white noise probability measure $\mathbb{P}$ on $\mathcal{F}$. 
This measure is  typically defined via the property
 \[
 \E  ( e^{i \langle \cdot, \phi \rangle  } )
 = e^{- \frac{1}{2} \| \phi \|^2_{L^2(\R^n)} }, \quad \phi \in \mathcal{S}(\R^n),
 \]
 where $\E ( e^{i \langle \cdot, \phi \rangle  } ) $ stands for the characteristic function
 \[
 \E  ( e^{i \langle \cdot, \phi \rangle  } ) = \int_{\mathcal{S}'(\R^n) } e^{i \langle \omega , \phi \rangle }  \mu (d\omega). 
 \]
The probability space $(\mathcal{S}'(\R^n),   \mathcal{F} , \mathbb{P} )$ is called the white noise probability space. 
Each $\omega \in \mathcal{S}'(\R^n)$ represents a single realization $\langle \omega, \cdot \rangle$ of $n$-parameter white noise. See e.g. \cite{hida-lectures-book, kuo-book} for more details on the topic.

%\begin{remark}
%Alternatively, one may choose $\mathcal{F}$ to be the %Borel $\sigma$-algebra $\mathcal{B}( %\mathcal{S}'(\R^d))$. 
%\end{remark}

%\footnote{J: Pitäisikö tähän laittaa myös jokin %yleisviite, missä näitä käsitelty?}

 Let us write $W_\phi := \langle   \cdot, \phi \rangle  $ for $\phi \in \mathcal{S}({\R^n})$. 
It follows from the definition above that white noise admits the isometry 
 \begin{equation}\label{itolike}
 \E W_\phi^2  =  \| \phi \|^2_{L^2}
 \end{equation} 
 (cf. {It\^o} isometry). Consequently, $\E(W_\phi W_\psi ) = \langle \phi, \psi \rangle_{L^2}$. 
Combining this with the zero mean $\E W_\phi = 0 $ yields $\text{Cov} = \text{id}$, where $\text{Cov} : \mathcal{S}( {\R^n})  \to \mathcal{S}'( {\R^n})$ stands for the covariance
 \[
  \langle \text{Cov} (\phi ), \psi \rangle := \E \Big(  ( W_\phi-   \E W_\phi ) (  W_\psi - \E W_\psi) \Big) .
 \]

\begin{remark}
It follows from \eqref{itolike} that $W_\phi $ can be defined as a $L^2 (\mathbb{P})$-function for every $\phi \in L^2 ( \R^d)$ by considering the limit $W_\phi := \lim_{n \to \infty} W_{ \phi_n} $ in $L^2(\mathbb{P})$, where $\phi_n$ stand for Schwartz functions converging to $\phi$ in $L^2 ( \R^d)$. 
This extension can be used to define Brownian motion: 
For simplicity, let us consider the 1-parameter case  $d=1$.
The multi-parameter construction is analogous. 
Let $ \chi_{[0,t)} $ for $t \geq 0$ be the characteristic function of the interval $[0,t)$. It is straightforward to check that the process $B_t  := W_{\chi_{[0,t)}  } $ is a single-parameter Brownian motion. 
White noise, as a stochastic process over $t \geq 0$, can therefore be thought as the distributional derivative $\partial_t B_t$. This is sometimes taken as a definition of white noise. Notice that the derivative $\partial_t B_t$ does not exist in the classical sense, in fact, Brownian motion is nowhere differentiable. 
\end{remark}

\section{The initial-boundary value problem}\label{IBVP}

In this section, $A$ stands for a  strictly positive smooth function $A : \R \to (0,\infty)$ with $\text{supp}(A-1) \subset (0,\infty)$. 
Later, we shall take $A$ to be admissible which is a stronger condition. 
%{In this section, $A$ stands for an admissible function $A : \R \to \R$ (see Definition \ref{def:admissible}). }
Define the {(1+1)-}dimensional wave operator 
\[
\square_A : \mathcal{D}'(  \R^2 )  \to   \mathcal{D}'(   \R^2  ) 
\]
by 
\[
\square_A u(x,t)   :=  \partial_t^2u (x,t)  - \frac{1}{A(x)}  \partial_x ( A(x)  \partial_x u(x,t)   ).
\]
Denote $\dot\R^2 := \R^2 \setminus \{0\}$.
The principal symbol (see e.g. \cite{grigis-sjostrand}) of $\square_A$ is represented (away from the origin) by the function 
\[
 p(x,t,\xi_x,\xi_t) =   - \xi_t^2   +     \xi_x^2 
\]
on $ (x,t,\xi_x,\xi_t) \in \R^2 \times \dot\R^2$. 
The associated characteristic variety is the null covector bundle
\[
\Sigma := \ker p \subset \R^2 \times \dot\R^2. 
\]
Let $f \in \mathcal{D}' (  \R )$ such that $\text{supp} (f) \subset (t_0,\infty)$. 
Fix $\mu$ to be either $0$ or $1$.
We consider the following initial-boundary value problem:
\begin{align}
&\square_A u (x,t)= 0 \quad \text{for} \quad (x,t) \in  (0,\infty) \times  \R ,\label{ibvp1}  \\
&u(x,t)  = 0 ,  \quad \text{for}  \quad (x,t) \in  (0,\infty)\times (-\infty,t_0), \label{ibvp2} \\
&\partial_x^\mu u (x,t) \big|_{x  = 0}  = f(t), \quad \text{for}  \quad t \in \R.  \label{ibvp4}
\end{align}
 The choice of $\mu$ determines whether the boundary condition \eqref{ibvp4} is of the Dirichlet ($\mu=0$) or Neumann  ($\mu=1$) type. 
For the problem to make sense, we need to fix an appropriate class of $u$. The issue is that the boundary restriction above is not well defined, a priori, for general $u \in \mathcal{D}'( (0,\infty) \times  \R )$. 
This can be resolved by requiring that $u$ lies in the subspace $\mathcal{D}_{\Sigma}'(\R^2)$.
%\[
%\mathcal{D}_{\Sigma}'(\R^2) := \{ u \in %\mathcal{D}'( \R^2 ) :  WF(u)  \subset  \Sigma %\}.  % \subset  \mathcal{D}'( (0,\infty) \times  %\R ).
%\] 
Indeed,  the inclusion
\begin{align}
  \iota_{x=0} : \R \to \R^2, \quad  \iota_{x=0} (t) := (0,t) \in \R^2,
  \end{align}
  generates (see \cite[Theorem 2.5.11']{hormanderFIO1}) a sequentially continuous pull-back 
\begin{align}
\iota^*_{x=0} : \mathcal{D}_{\Sigma}'(\R^2)  \to \mathcal{D}'(\R) 
%\ \subset \ \mathcal{D}' \big((0,\infty) \big)
\end{align}
which can be applied to extend the standard boundary trace operator $u \mapsto u |_{x=0}$ as well as the Neumann data $u \mapsto \partial_x u |_{x=0}$. 
Analogously, the restriction $ u \mapsto u |_{t=0}$ extends continuously to $ \mathcal{D}_{\Sigma}'(\R^2)$ and we can consider the Cauchy data $(u|_{t=0}, \partial_t u |_{t=0})$ for $u$ in this class of distributions. 

Define $\mathcal{D}'(\R)_{t_0} := \{ u \in \mathcal{D}'(\R) : \supp (u) \subset (t_0, \infty) \} \subset \mathcal{D}'(\R) $ and equip it with the subspace topology inherited from $\mathcal{D}'(\R)$. 
Here the topology of $\mathcal{D}'(\R)$ is the usual inductive limit topology described in Section \ref{subsection_test-functions}. 

\begin{proposition}[Existence of solutions]\label{exist-pro}
Let $f \in \mathcal{D}'(\R)_{t_0}$ 
%be supported in $(t_0,\infty)$ 
and choose $\mu$ to be either $0$ or $1$.   Then there exists $u \in \mathcal{D}_{\Sigma}'(\R^2) $ that satisfies \eqref{ibvp1}-\eqref{ibvp4} and depends continuously on $f$ in $\mathcal{D}'(\R)_{t_0}$.
%\footnote{J: Pitäisi kirjoittaa, että missä mielessä(topologiassa).}. 
\end{proposition}
%\footnote{J: (10)--(12) ihan vaan tyyliin liittyvä juttu?}

The proof relies on the following { well-known lemma. It is proven in many cases with regularity, see e.g. \cite{romanov20_recon_princ_coeff_damped_wave,romanov20_recov_poten_damped_wave_equat}. However we verify the usual steps here to make sure that no hidden issues arise when the boundary term is a distribution.}

\begin{lemma}\label{weaker-exist}

Let $f \in \mathcal{E}' (\R)$ be supported in $(0,\infty)$ and let $\mu$ be either $0$ or $1$. 
Fix small $s,\delta>0$ such that $\text{supp} (f) \subset (s,\infty) $ and $A|_{(-\infty , \delta)} = 1$. 
Then  
there is a solution $u = u_{f} \in \mathcal{D}_{\Sigma}'(\R^2) $ to
\begin{align}
&\square_A u = 0 \quad \text{in} \quad  (0,\infty) \times  \R, \label{ibvp11}  \\
&u = 0 ,  \quad \text{in} \quad  (0,\infty) \times (-\infty,0), \label{ibvp22} \\
% \quad  \{ (x,t) \in (0,\infty)\times \R : t < x+s \} 
&\partial_x^\mu u  \big|_{x  = 0}  = f, \quad \text{in}  \quad (-\infty, \ 2\delta + s) . \label{ibvp44}
\end{align}
that vanishes in the wedge $W := \{ (x,t) \in \R^2 : |t| <x +s\}$ and
%Moreover, the solution {can be chosen} such that {it} 
depends continuously on $f$  with respect to the topology (see Section \ref{subsection_test-functions}) of $\mathcal{E}'(\R)$.
%\footnote{J: Pitäisi kirjoittaa, että missä mielessä(topologiassa).}
{ Moreover, the solution $u=u_f$ does not depend of the choice of $s \in (0, \min \supp (f))$ and is linear as a function of elements $f \in \mathcal{E}'(\R)$ supported in $(0,\infty)$.
%\[
%\alpha_1 u_{f_1}+ \alpha_2 u_{f_2} = %u_{\alpha_1f_1+\alpha_2f_2} 
%\]
%for any  $\alpha_1,\alpha_s \in \R$ and  $f_1,f_2 \in %\mathcal{E}'(\R)$ supported in $(s,\infty)$.
}  
\end{lemma}

Before proving the lemma, let us show how to derive Proposition \ref{exist-pro} from it. 

\begin{proof}[Proof of Proposition \ref{exist-pro}]
%{(ANTTI EDITING THIS)}
We may assume that $t_0 = 0$. 
%Let us consider the compactly supported case $f \in \mathcal{E}'(\R)$ at first. 
Let { $\tau_T \in C_c^\infty(\R)$ stand for a cut-off that equals $1$ near $[-T, T]$, $T>0$ and define $f_0^T := f$. Let $ u_{0}^T$ be the solution in Lemma \ref{weaker-exist} with $\tau_T f_0^T$ as  the boundary value. }
Then $\partial_x^\mu u_0^{{T}} |_{x = 0} =   {\tau_T} f $ in $ (-\infty, 2\delta + s) $. 
%{for $T >  2\delta + s$. }. 
%
Set $
f_j^T :=    f_{j-1}^T- \partial_x^\mu u_{j-1}^T |_{x=0} 
$ and define $u_{j}^{{T}}$, $j=1,2,3,\dots$ recursively as the solution in Lemma \ref{weaker-exist} but with the function ${\tau_T} f_j^T$  substituted for $f$  
%($f$) 
and $s_j := s_{j-1} + \delta =  s + j \delta$ substituted for $s$.
%{ Here $\tau_L \in C^\infty(\R)$ stands for %a cutoff that equals $1$ near $(-\infty, L]$. It %ensures that the boundary value $\partial_x^\mu u_{j} %|_{x=0}$ is compactly supported by cutting off the %tail of it.  
%{Observe that 
%the identity $\partial_x^\mu u_{j} |_{x=0} =  f_{j-1}- \partial_x^\mu u_{j-1} |_{x=0}$
%and $\partial_x^\mu u_{j} |_{x=0} = 0$ 
%holds in $(-\infty,2\delta +s_j) $ and 
%$\text{supp} (\partial_x^\mu u_{j}^{{T}} |_{x=0} ) 
%\subset (s_j,\infty)$. 
%}
Since $u_j^{{T}}$ vanishes in the wedge
\[
W_{s_j} := \{ (x,t) \in \R^2 : |t| <x +s_j\} 
 \]
 and $\lim_{j \to \infty} s_j = \infty$, 
the sum $u^{{T}} :=  \sum_{j \geq 0} u_j^{{T}}$ is well defined in $\mathcal{D}'_\Sigma(\R^2)$. Indeed, only finite number of terms $u_j^{{T}}$ are non-zero in a relatively compact neighbourhood. Moreover, each $u_j^{{T}}$ depends  continuously on $\tau_T f$ {which further depends continuously on $f$. }
Hence, $u^{T}$ is continuous as a function of $f \in {\mathcal{D}}_{0}'(\R^n)$.  
T

{
Let $K \subset \R^2 $ be compact. 
Let us show that $u^T = u^{T_K}$ near $K$ for 
$T \geq  T_K := 2\min \{ r \geq 0 : K  \subset W_r \}$, where $W_r  $ stands for the wedge 
\[
W_r := \{ (x,t) \in \R^2 : |t| < x+r\}.
\]
This implies that the weak limit $\lim_{T \to \infty} u^T$ exists and remains continuous in $f$. 
Define $v:= u^T-u^{T_K}$, $T \geq T_K$. 
Let $N$ be so large that  $T_K < s_N $ and therefore $W_{T_K} \subset W_{s_N}$. 
Then $u_j^T $  vanishes near $W_{T_K}$ for every $j > N$ and hence $u^T = \sum_{0\leq j \leq N} u_j^T$ near $W_{T_K} $. 
%In particular, this holds near $K$. 
Moreover, 
\[
\partial_x^\mu u^T |_{x=0} = \sum_{0\leq j \leq N} \partial_x^\mu u_j^T|_{x=0} \quad \text{near} \quad  (-\infty,T_K].
\]
%The boundary value satisfies 
%\begin{align}
%\partial^\mu_x v|_{x=0} 
%%= \sum_{0\leq j\leq N}  \partial^\mu_x u_j^T|_{x=0}-%\partial^\mu_x u_j^{T_K}|_{x=0} 
%= \sum_{0\leq j\leq N} (\tau_{T} f_j^{T}- \tau_{T_K} %f_j^{T_K})  \quad  \text{in} \quad   (-\infty, s_N + 2 %\delta) . \label{e1qq}
%\end{align}
%in $(-\infty,T_K +2\delta)$ 
We have 
\begin{align}
f_N^T &= f_{N-1}^T - \partial_x^\mu u_{N-1}^T|_{x=0}   \\
&= f_{N-2}^T - \partial_x^\mu u_{N-2}^T|_{x=0} - \partial_x^\mu u_{N-1}^T |_{x=0}  \\
\vdots 
\\
&=  f_0^T - \sum_{0 \leq j\leq N-1} \partial_x^\mu u_{j}^T  |_{x=0} 
\\
&=   f -  \partial_x^\mu u^T |_{x=0} + \partial_x^\mu
u_N^T|_{x=0} \quad \text{near} \quad  (-\infty,T_K]
\\
&=   f -  \partial_x^\mu u^T |_{x=0} + \tau_T f_N^T \quad \text{near} \quad  (-\infty,T_K].
\end{align}
Hence,
\[
 \partial_x^\mu u^T |_{x=0}  = f   -  (1-\tau_T) f_N^T  \quad  \text{near} \quad   (-\infty, T_K].
\]
Since $T \geq T_K$, the cut-off $(1-\tau_T)$ in the latter term vanishes near $(-\infty, T_K]$ and we arrive at 
\begin{equation}\label{qoo}
 \partial_x^\mu u^T |_{x=0}  = f   \quad  \text{near} \quad   (-\infty, T_K].
\end{equation}
By the same argument, 
\[
 \partial_x^\mu u^{T_K} |_{x=0} =  f-  (1-\tau_{T_K}) f_N^{T_K}  = f  \quad  \text{near} \quad   (-\infty, T_{K}].  
\]
Thus, 
\[
\partial_x^\mu v|_{x=0}  = (\tau_T- \tau_{T_K}) f   
%-  (1-\tau_T) f_N^T+  (1-\tau_{T_K}) f_N^{T_K}  
\quad  \text{near} \quad   (-\infty, T_K].
\]
As the cut-off $(\tau_T- \tau_{T_K})$ on the right hand side is supported in $(T_K , \infty)$, we obtain  
\[
\text{supp} (\partial_{x}^\mu v |_{x=0}) \subset (T_K , \infty). 
\]
By the linearity statement in Lemma \ref{weaker-exist}, $v$ is a wave with boundary value as above.
Moreover, $v$ vanishes in the wedge 
\[
 \{ (x,t) \in \R^2 : |t| <x +T_K\} .
 \]
In particular, it vanishes near $K$. 
Thus, 
\[
u^T = u^{T_K} - v = u^{T_K} \quad \text{near} \quad K.
\]

As shown above, the weak limit $u:= \lim_{T \to \infty} u^T $ exists and  is continuous in $f$. 
Moreover, it equals $u^{T_K}$ near compact $K \subset (0,\infty) \times \R$. In particular,
\[
\langle u ,  \varphi  \rangle  = \langle u^{T} , \varphi \rangle = \langle u^{T_K} , \varphi \rangle , 
\]
for    $\varphi \in C_c^\infty((0,\infty) \times \R)$ and $T \geq T_K$,  
where $K$ is such that $\supp(\varphi) \subset K$. 
Since $\square_A$ is a local operator and $ \square_A u^{T_K} = 0$, we have 
\[
\langle \square_A u ,  \varphi  \rangle  = \langle u , \square_A \varphi \rangle 
= \langle u^{T_K} , \square_A \varphi \rangle 
= \langle \square_A u^{T_K} , \varphi \rangle  = 0. 
\]
That is, $\square_A u =0$. 

Finally, let us show that $\partial_x^\mu u|_{x=0} = f$. Fix arbitrary $\phi \in C_c^\infty(\R)$ and 
%and fix compact $K' \subset \R$ such that 
%$\supp(\phi) \subset K' $. 
set $K :=  [-1,1] \times \supp(\phi) $. 
With this choice, 
%$\supp(\phi) \subset (-\infty, T_K]$ and 
$\partial_x^\mu u^T |_{x= 0}  = f$ holds near $(-\infty, T_K] $ for $T \geq T_K$ by \eqref{qoo}. 
In particular, it holds near the support of $\phi$. 
By continuity of  $\partial_x^\mu$ and the boundary restriction, 
\[
\langle  \partial_x^\mu u |_{x= 0} , \phi \rangle 
= \langle  \partial_x^\mu (\lim_{T \to \infty} u^T ) |_{x= 0} , \phi \rangle 
 = \lim_{T \to \infty} \langle  \partial_x^\mu  u^T  |_{x= 0} , \phi \rangle 
  = \lim_{T \to \infty}\langle f, \phi \rangle = \langle f, \phi \rangle.
\]
This finishes the proof. 
%\[
%\partial_x^\mu u^T |_{x= 0}  = f
%\]

%in the Fr\'echet space $C_K^\infty((0,\infty\times \R) = \{ \varphi \in C^\infty((0,\infty\times \R) : \supp ( \varphi) \subset K\}. $

%which is supported in $(T_K,\infty)$. 
%Moreover, $v$ vanishes in the wedge $W_K := \{ (x,t) \in %\R^2 : |t| <x +T_K\}$. Observe that $K \subset W_K$. 
%Hence, $v = 0$ near $K$ and  
%using $u^T = v + u^{T_K}$ 
%we find that $u^T = u^{T_K}$ near $K$. 
}

\end{proof}

\begin{proof}[Proof of Lemma \ref{weaker-exist}]
%{ (I'd like to make this proof a bit shorter and maybe move it to the appendix --Antti)}
 Let $f_0,f_1 \in \mathcal{D}'(\R)$.
 %KORJATTU \footnote{J: Pitäisi olla $f_0, f_1$?}
By  \cite[Th. 5.1.6.]{duistermaat-FIOs}, there is a unique solution $u \in \mathcal{D}_\Sigma' (\R^2)$ to the Cauchy problem 
\begin{align}
&\square_A u = 0 \quad \text{on} \quad \R^2, \\
& u|_{t=0}  = f_0 \\
&\partial_t u |_{t=0} = f_1.
\end{align} 
Moreover, the solution depends continuously on the data $(f_0,f_1)$. Indeed, the additional condition \cite[(5.1.10)]{duistermaat-FIOs} in the theorem is met.  
%\footnote{J: Onko tämä selvää?} 
{Below we apply this
with specific choice of $f_0$, $f_1$. The idea is that, with the right choice of Cauchy data, the solution of the  Cauchy problem solves the initial-boundary value problem \eqref{ibvp11}-\eqref{ibvp44} away from $(0,\infty)\times (-\infty,-s]$. By cutting off the component of  $\supp (u)$ in  $(0,\infty)\times (-\infty,-s]$ we obtain a solution in the half-space $(0,\infty)\times \R$. Hence, the Cauchy problem in $\R^2$ can be used to construct solution to the initial-boundary value problem \eqref{ibvp11}-\eqref{ibvp44} in the half-space.}

Let $\phi \in C^\infty_c( \R) $ be a mollifier and denote $\phi_\varepsilon =  \frac{1}{\varepsilon} \phi \left( \frac{x}{\varepsilon} \right)$. 
 Let $u_\epsilon   \in C^\infty(\R^2)$ be the {global} solution to the following Cauchy problem:
 \begin{align}
&\square_A u_\varepsilon = 0 \quad \text{on} \quad \R^2, \label{cauchy1} \\
& u_\varepsilon |_{t=0}  = (-1)^\mu \phi_\varepsilon *   \tau^* f^{(-\mu)} , \\  
&\partial_t u_\varepsilon  |_{t=0} =  (-1)^\mu \phi_\varepsilon *  \tau^*  f^{(1-\mu)} . \label{cauchy3}
\end{align} 
where $\tau : \R \to \R$, $\tau (x) := -x$, $f^{(0)} (x) := f(x) $, $f^{(1)} (x):= \partial_x f(x) $ and $f^{(-1)}$ is the forward propagating inverse derivative, defined by
\begin{align}
\langle f^{(-1)}, \phi \rangle  := 
-\langle f, \tilde\phi  \rangle, \quad  
   \tilde\phi(x) := \int_{-\infty}^{x } \phi(r) dr  ,
\end{align}
for every $\phi \in C_c^\infty(\R)$. 
Notice that this is well defined since $f$ is compactly supported. Moreover, $f^{(-\mu)} $ is supported in $(s,\infty)$.
The solution $u_0$ associated with the data $(f_0,f_1) = ((-1)^\mu\tau^*f^{(-\mu)}, (-1)^\mu \tau^*f^{(1-\mu)})$ coincides with the limit $  \lim_{\varepsilon \to 0} u_\varepsilon$.
{ Let us shortly explain why we choose this wave. First, suppose  we have a trivial profile $A=1$. Then the explicit expression for  the wave $u_\varepsilon$ can be derived by integrating along the null lines $\{ x = c\pm t\}$   (cf. the local version of this argument below). In that case, we would find, after applying a cut-off, that $u_\varepsilon$ solves $\square_A u = 0$, $\partial^\mu_x u |_{x=0} = \phi_\varepsilon * f$ and 
$u = 0 = \partial_t u$ for $t<0$ for small $\varepsilon>0$.  
Taking the limit $\varepsilon \to 0$ would therefore give the solution to \eqref{ibvp11}-\eqref{ibvp44} (even for $\delta= \infty $) by continuity. The mollification
 in the data was introduced so that the integration along null curves would be meaningful.  
Somewhat similar argument could be made  directly without the mollification. 
 Notice that the Cauchy problem itself  does not require smoothness. 
Let us now return to the case where $A$ is non-trivial in $(\delta, \infty)$. 
Although $A=1$ holds near the t-axis $\{x = 0\}$, the preceding argument is not valid anymore. The issue is the possibility of echoes originating from the region $x>\delta$. 
Nevertheless, such echoes reach the point $x=0$ only after a delay: the initial wave must first
propagate into the region, and any scattered signal must then return to $x=0$.
Consequently, the boundary value $\partial^\mu_x u |_{x=0} = \phi_\varepsilon * f$ is preserved for early times $t>0$. More precisely, this holds for  $t \in (-\infty,2\delta +s)$ provided that $\varepsilon>0$ is sufficiently small. 
}
%\footnote{J: Liittynee tuohon aikaisempaan footnoteen.} 

{Let us now do this in detail. As implied above,} 
the operator $\square_A$ reduces to the trivial wave operator ${ \square_1 =} \square =  \partial_t^2 - \partial_x^2$ in the half space $ (-\infty,\delta ] \times \R $ for any sufficiently small $\delta >0$.
It is easy to check that $ \square $ equals  $\frac{1}{4} \partial_{0} \partial_{1}$ in the coordinates $(z_0,z_1) :=  (t-x,  t+x)$. 
Integrating twice, one deduces that the general smooth solution to  $\frac{1}{4} \partial_{0} \partial_{1} v = 0$ is 
\[
 v(z_0,z_1) = F(z_0 ) + G(z_1)  = (F \otimes 1) (z_0,z_1 ) + (1 \otimes G) (z_0,z_1) ,
 \]
 where $F,G \in C^\infty(\R)$. 
Moreover, if  $ v_\varepsilon  = F_\varepsilon \otimes 1 + 1 \otimes G_\varepsilon$ for $\varepsilon >0$  is a family of such solutions and the limit $\lim_{\varepsilon \to 0} v_\varepsilon$ exists in $\mathcal{D}'(\R^2)$, we may choose 
$F_\varepsilon $ and $G_\varepsilon $ such that they converge individually to some distributions $ F_0  $ and $G_0 $. 

%{Since the global solution  $u_\varepsilon$ obeys $\square u_\varepsilon = 0$ in}
Hence, { the global solution $u_\varepsilon \in C^\infty ( \R^2)$ must be of the form}
\begin{equation}\label{wqi}
u_\varepsilon (x,t) =  F_\varepsilon (t-x ) + G_\varepsilon (t+x) 
\end{equation} 
in $ (-\infty,\delta] \times \R $ for some $F_\varepsilon,G_\varepsilon $ that converge in $  \mathcal{D}'(\R)$ as $\varepsilon \to 0$. 
{Notice that this expression holds even though the functions $F_\varepsilon,G_\varepsilon$ cannot be solved from local information alone.}
Formally,  we have
\begin{equation}\label{wqi1}
u_0 (x,t) =  F_0 (t-x ) + G_0 (t+x) 
\end{equation} 
for $(x,t) \in (-\infty,\delta) \times \R$. For the rigorous expression, 
define $\rho_\pm :  \R^2  \to \R $, $\rho_\pm (x,t) := t\pm x$ and  extend the pull-backs $\rho^*_\pm$ into continuous operators $\rho_\pm^* :  \mathcal{D}'( \R ) \to \mathcal{D}'( \R^2 )$  by applying the last identity in 
\[
\langle \rho^*v , \phi \rangle =  \int_{\R^2}  v(t\pm x) \phi(x,t)  dt dx  =  \int_{\R^2}  v(s ) \phi(x,s \mp x )  ds dx .
\]
That is, $\rho_\pm^*v = (\sigma_\pm )_* ( 1 \otimes v )$ where $\sigma_\pm : \R^2 \to \R^2$ is the diffeomorphism $\sigma_\pm (x,s) = (x,s \mp x)$. 
The identity \eqref{wqi} then reads 
\[
u_\varepsilon 
=  \rho_-^* F_\varepsilon + \rho_+^* G_\varepsilon  , \quad \text{in } \quad  (-\infty,\delta] \times \R
\] 
and taking the limit yields 
\begin{equation}\label{s11}
u_0
=  \rho_-^* F_0 + \rho_+^* G_0, \quad \text{in } \quad  (-\infty,\delta) \times \R
\end{equation}
by the continuity of $ \rho_\pm^* $. This is the formal identity \eqref{wqi1} in a rigorous form. 
The two terms,  $  \rho_-^* F_0 \in \mathcal{D}_{\Sigma^\rightarrow}' ( (-\infty,\delta) \times \R )$ and $  \rho_+^* G_0  \in \mathcal{D}_{\Sigma^\leftarrow}' ( (-\infty,\delta) \times \R )$, represent the components of $u_0$ propagating in the {respective} characteristic component
\[
\Sigma^{\rightarrow} := \{ (x,t, -  \lambda , \lambda) : \lambda \in \dot\R\}, \quad \Sigma^{\leftarrow} := \{ (x,t,  \lambda , \lambda) : \lambda \in \dot\R\}
\]
towards the infinities $(   \pm \infty , \infty)$ (i.e. to the right or left).
In fact, we may split 
\[
u_0 = u_0^\rightarrow + u_0^\leftarrow, \quad  u_0^{\rightarrow} \in \mathcal{D}_{\Sigma^{\rightarrow }}' ( \R^2 ), \quad  u_0^{\leftarrow} \in \mathcal{D}_{\Sigma^{\leftarrow }}' ( \R^2 )
\]
in the whole $\R^2$ such that $u_0^\rightarrow =  \rho_-^* F_0$ and $u_0^\leftarrow =  \rho_+^* G_0$ in $(-\infty , \delta) \times \R$. 
\begin{figure}[H]
\begin{center}
\begin{tikzpicture}[scale = 0.8]
\fill[lightgray] (1,5) -- (-5,5) -- (-5,-5) -- (1,-5) --cycle;
\fill[gray] (-1,0) -- (4,5) -- (4,-5) -- cycle;
\draw[->] (-6,0) -- (5,0) node[anchor=north west] {$x$ axis};
\draw[->] (-0.5,-6) -- (-0.5,7) node[anchor=south] {$t$ axis};
%\filldraw[black] (-1,0) circle (2pt) node[anchor=west]{x=-s};
\draw[dashed] (-1,-6) -- (-1,6) node[anchor=south east] {$x=-s$ };
\draw[dashed] (1,-6) -- (1,6) node[anchor=south] {$x=\delta$ };
%\draw[dashed] (-6,2) -- (5,2) node[anchor=south west] {$t=s+\delta$ } ;
%\draw[dashed] (-3,-1) -- (5,-1) node[anchor=north east] {t=-s};
 \node at (2.5, -1)   (a) {\large $W$};
   \node at (-3, -4)   (a) {\large $(-\infty, \delta)\times \R$};
\end{tikzpicture}
\caption{A visualization of the sets $(-\infty, \delta) \times \R$  and $W$. The light gray area represents the half space $(-\infty, \delta) \times \R $. The wedge $W$ is the darker shape on the right.}\label{captionW}
\end{center}
\end{figure}
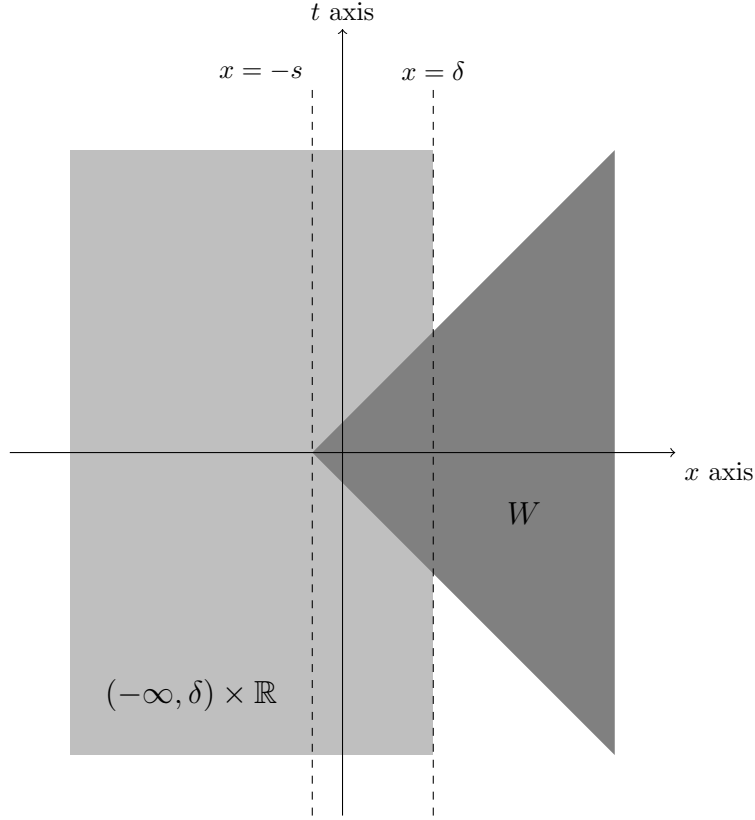

Let $s>0 $ be such that 
$
\text{supp} (f) \subset (s ,\infty),
$ 
as in the assumptions. 
Then the Cauchy data of $u_0$ is supported in $(-\infty ,-s)$ and 
$u_0$ vanishes in the wedge  $W:=  \{ (x,t) :   |t| - s < x  \}$ (see Figure \ref{captionW}) by   \cite[Th. 5.1.6]{duistermaat-FIOs}. 
This implies that 
 $u_0^\rightarrow = - u_0^\leftarrow $ in $W $. 
Since the wave front sets of these terms do not overlap, they must both be smooth in $W$ for this to hold. In $W \cap (0,\delta) \times \R$, 
they then reduce into $u_0^\rightarrow(x,t) =  F_0(t-x)$ and $u_0^\leftarrow(x,t) =  G_0(t+x)$, where $F_0$ and $G_0$ are smooth. 
In the coordinates $z_0 = t-x$ and $z_1 = t+x$, this reads $u_0^\rightarrow(z_0,z_1) =  F_0(z_0)$ and $u_0^\leftarrow(z_0,z_1) =  G_0(z_1)$.
Setting  $F_0(z_0)  = - G_0(z_1)$ in $W \cap (0,\delta) \times \R$ implies $F_0(z_0) = C = - G_0(z_1)$ there. 
We can assume $C=0$ by redefining the functions $F_0  $ and $G_0 $. Indeed, the opposing constants $\pm C$ may be added to them without changing $u_0$. 
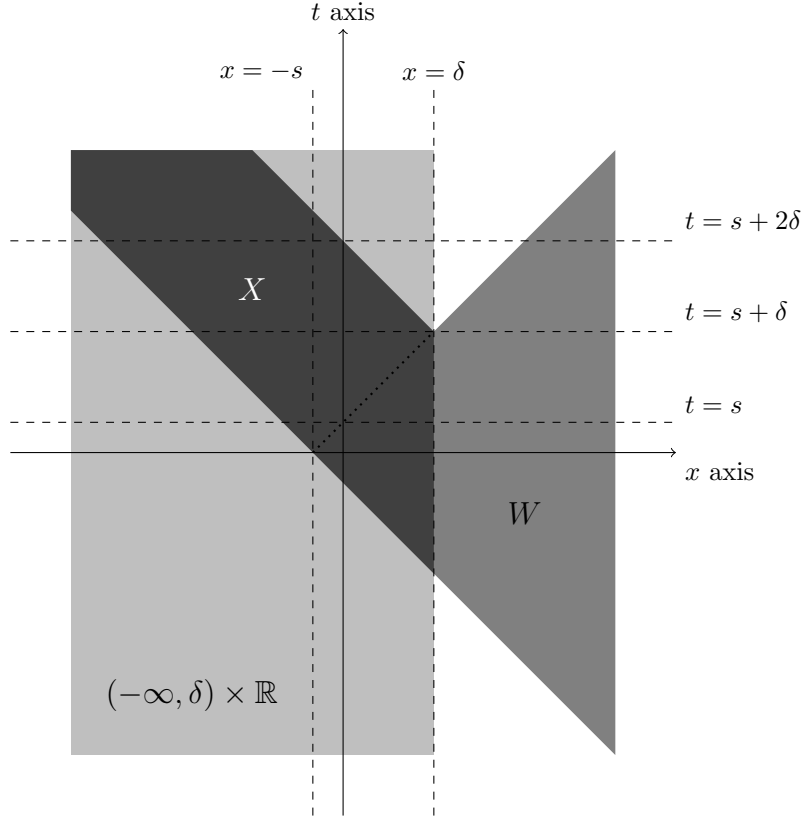
\begin{figure}[ht]
\begin{center}
\begin{tikzpicture}[scale = 0.8]
\fill[lightgray] (1,5) -- (-5,5) -- (-5,-5) -- (1,-5) --cycle;
\fill[gray] (-1,0) -- (4,5) -- (4,-5) -- cycle;
\fill[darkgray] (1,-2) -- (1,2) -- (-2,5) -- (-5,5) -- (-5,4) -- cycle;
\draw[->] (-6,0) -- (5,0) node[anchor=north west] {$x$ axis};
\draw[->] (-0.5,-6) -- (-0.5,7) node[anchor=south] {$t$ axis};
%\filldraw[black] (-1,0) circle (2pt) node[anchor=west]{x=-s};
\draw[dashed] (1,-6) -- (1,6) node[anchor=south] {$x=\delta$ };
\draw[dashed] (-1,-6) -- (-1,6) node[anchor=south east] {$x=-s$ };
%\draw[dashed] (-3,-1) -- (5,-1) node[anchor=north east] {t=-s};
\draw[dotted, thick] (-1,0) -- (1,2) ;
\draw[dashed] (-6,0.5) -- (5,0.5) node[anchor=south west] {$t=s$ } ;
\draw[dashed] (-6,2) -- (5,2) node[anchor=south west] {$t=s+\delta$ } ;
\draw[dashed] (-6,3.5) -- (5,3.5) node[anchor=south west] {$t=s+2\delta$ } ;
 \node at (2.5, -1)   (a) {\large $W$};
  \node at (-3, -4)   (a) {\large $(-\infty, \delta)\times \R$};
   \node[white] at (-2, 2.7)   (a) {\large $X$};
\end{tikzpicture}
\caption{A visualization of $X$  in relation to $W$ and $(-\infty, \delta)\times \R$. 
The set $X$ is shown in dark gray.  The medium and light gray areas represent  $W$ and $(-\infty,\delta) \times \R$, respectively.  }\label{captionX}
\end{center}
\end{figure}
%
%\newpage
%
With this choice, $u_0^\leftarrow = 0$ in $W \cap (0,\delta) \times \R$  and the invariance of $G_0(t+x)$ along the rays 
\[
R_c = \{ (x,t) \in (-\infty,\delta) \times \R :  t + x =c\}, c \in \R, 
\]
implies that $u_0^\leftarrow$
vanishes in 
\[
X:= \cup \{ R_c : c \in \R, \ R_c \cap W \neq \emptyset\} = \{  (x,t) \in (-\infty,\delta) \times \R :  -x-s < t < 2 \delta + s -x \}.
\] 
This set is shown in Figure \ref{captionX}. 
%It is straightforward to check that $X= \{  (x,t) \in (-\infty,\delta) \times \R :  -x-s < t < 2 \delta + s -x \}$. 
By   \cite[Th. 5.1.6]{duistermaat-FIOs}, the solution to the Cauchy problem is unique in the wedge 
\[
V =  \{ (x,t) \in (-\infty,\delta) {\times \R} : x < -|t| + \delta\}.
\] This set is shown in Figure \ref{captionV}. 
\begin{figure}[ht]
\begin{center}
\begin{tikzpicture}[scale = 0.8]
\fill[lightgray] (1,5) -- (-4,5) -- (-4,-5) -- (1,-5) --cycle;
\fill[gray] (-1,0) -- (4,5) -- (4,-5) -- cycle;
\fill[darkgray] (1,0) -- (-4,5) -- (-4,-5) -- cycle;
\draw[->] (-5,0) -- (5,0) node[anchor=north west] {$x$ axis};
\draw[->] (-0.5,-6) -- (-0.5,6) node[anchor=south east] {$t$ axis};
%\filldraw[black] (-1,0) circle (2pt) node[anchor=west]{x=-s};
\draw[dashed] (1,-6) -- (1,6) node[anchor=south] {$x=\delta$ };
%\draw[dashed] (-3,-1) -- (5,-1) node[anchor=north east] {t=-s};
 \node[white] at (-2, -1)   (a) {\large $V$};
 % \node at (-2, 4.6)   (a) { $(-\infty, \delta) \times \R$};
   \node at (2.5, -1)   (a) {\large $W$};
\end{tikzpicture}
\caption{A visualization of $V$ in relation to the sets $W$ and $(-\infty, \delta ) \times \R$. The wedge $V$ is shown in dark gray.  The medium and light gray areas represent  $W$ and $(-\infty,\delta) \times \R$, respectively. }\label{captionV}
\end{center}
\end{figure}
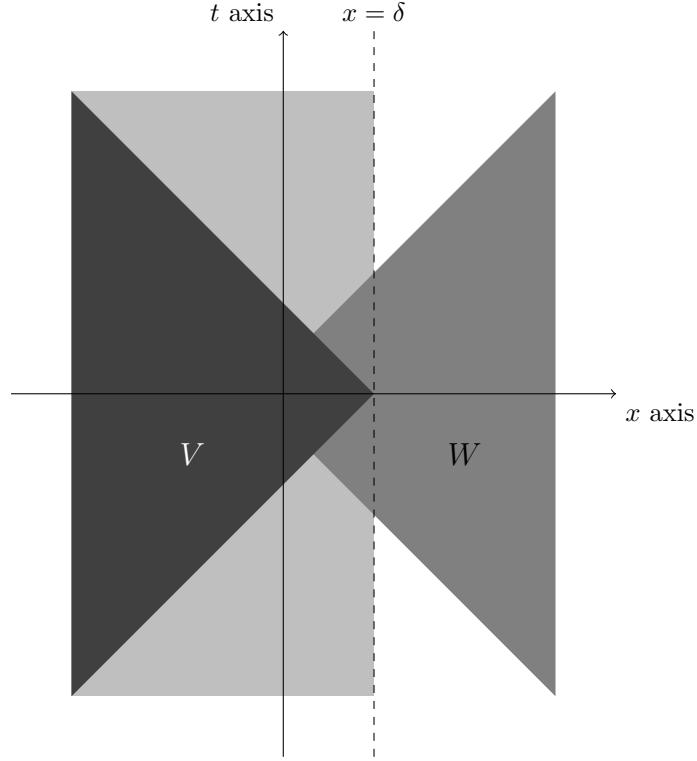
In our case, that solution is the function $(-1)^\mu \rho^*_- f^{(-\mu)} $. 
Hence, $u_0 = (-1)^\mu \rho^*_- f$ in $V$. 
This means that   $u_0^\rightarrow= (-1)^\mu \rho^*_- f^{(-\mu)} $ and $u_0^\leftarrow = 0$ in $V$. 
On the other hand, $u_0^\rightarrow= \rho^*_- F_0$ holds in $(-\infty,\delta) \times \R$ and applying the invariance 
of $\rho^*_- F_0$ along the rays $x-t =c$, we deduce that $u_0^\rightarrow =(-1)^\mu \rho^*_- f^{(-\mu)} $ holds in 
\[
X_{t > 0} := \{ (x,t) \in X : t > 0 \}
\]
too.
Indeed, the rays from $V \cap X$ fill $X_{t > 0}$ and we obtain the extension of $u_0^\rightarrow = (-1)^\mu \rho^*_- f^{(-\mu)} $ from $V \cap X$ to this set by the invariance. 
Since we already know that $u_0^\leftarrow$ vanishes in $X$,  we conclude that $u_0 = (-1)^\mu \rho^*_- f^{(-\mu)} $ does not only hold in $V$ but also in $V \cup X_{t >0}$. 
Recall that by  \cite[Th. 5.1.6.]{duistermaat-FIOs} the wave $u_0$ vanishes in $W \supset (0,\infty) \times \{0\}$. This further extends the identity $u_0 = (-1)^\mu \rho^*_- f^{(-\mu)} $ to 
$V \cup X_{t >0} \cup W$. 
Considering $u_0$ as a distribution in $\mathcal{D}'((0,\infty) \times \R)$, 
we may cut off everything  below $W \cap (0,\infty) \times \R$ without conflicting $\square_A u_0 = 0$. 
Indeed, $\square_A  u  = 0$ in  $(0,\infty) \times \R$ for $u := \chi u_0$, where $\chi \in C^\infty$ satisfies $\chi|_{t>-s/4} = 1$  and $\chi|_{t<-s/2} = 0$. 
Moreover, $u$ vanishes in $(0,\infty) \times (-\infty,0)$ since $u_0 = 0$ in $W$ and $\chi$ cuts off everything below it in $(0,\infty) \times \R$. 
As a conclusion, the identity $u = (-1)^\mu \rho^*_- f^{(-\mu)} $ holds in
\[
V \cup X_{t >0} \cup W \cup \ker \chi .
\]
This set contains an open neighbourhood of $\{0\} \times (-\infty, 2 \delta  + s)$ and hence 
\[
\partial^\mu_x u|_{x=0} = \partial^\mu_x ( (-1)^\mu \rho^*_-  f^{(-\mu)} )|_{x=0} =  f  
\]
for $t<2 \delta + s$. 

{
Finally, we note that the element $u=u_f$ above
does not change if we choose different \\ $s \in (0,\min \supp(f))$. Indeed, the Cauchy problem \eqref{cauchy1}-\eqref{cauchy3}, which was used to generate $u$, remain the same. 
Moreover, the linearity of $u$ in $f$ is also inherited from the Cauchy problem. 
}
\end{proof}
%\footnote{J: Saiskohan edelliseen todistukseen liittyen jonkun kuvan/kuvasarjan piirrettyä tukemaan todistuksen lukemista?}

%\newpage 

\begin{proposition}[Uniqueness of solutions]\label{uniq-pro}
Let $f \in \mathcal{D}'(\R)$ be supported in $(t_0,\infty)$.  Let $u_1,u_2 \in \mathcal{D}_{\Sigma}'(\R^2) $ be two solutions to (\ref{ibvp1})-(\ref{ibvp4}). 
Then $u_1 = u_2 $ in $(0,\infty) \times \R$. 
\end{proposition}
\begin{proof}
Throughout the proof, we may assume that $t_0 = 0$. 
Define $v = u_1- u_2$. 
Then $v$ obeys \eqref{ibvp1}-\eqref{ibvp4} with $0$ substituted for $f$.  Moreover, $v$ vanishes in $\{ (x,t) \in (0,\infty) \times \R : t < x \}$ by \cite[Th. 5.1.6]{duistermaat-FIOs} 
and the fact that $v=0$ in  $(0,\infty) \times (-\infty,0)$.
Let us first assume  that $v \in C^\infty ([0,\infty) \times \R)$. 
Then we may consider the energy $E: [0,\infty) \to [0,\infty)$, 
%\footnote{J: Pitäisikö kirjoittaa $\nabla_{x,t}$ merkinnän $d_{x,t}$ sijaan?}
\[
 E(t) := \frac{1}{2} \int_0^t |\nabla_{x,t}v(x,t)|^2 A(x) d x = \frac{1}{2} \int_0^t  [ (\partial_t v(x,t))^2 + (\partial_x v(x,t))^2 ]A (x) d x
\]
 over $x$ in the expanding interval $(0,t)$. 
The derivative of $E(t)$ is 
\begin{align}
E'(t) =   \int_0^t  [ \partial_t v(x,t)    \partial_t^2 v(x,t) + \partial_x v(x,t) \partial_t \partial_x v(x,t)  ] A(x) dx \\+ \frac{1}{2}   [ (\partial_t v(t,t))^2 + (\partial_x v(t,t))^2 ] A (t).
\end{align}
The latter term must be zero since $v$ is smooth and vanishes on $\{ (x,t) \in (0,\infty) \times \R : t < x \}$.
In the first term, we integrate by parts. This gives 
\begin{align}
E'(t) 
=  \int_0^t   \partial_t v(x,t)  (   \square_A v(x,t)  ) A(x)dx \\
+A(t) \partial_t v(t,t)  \partial_x v(t,t) - A(0) \partial_t v(0,t)  \partial_x v(0,t)  
\end{align}
The first term vanishes since $  \square_A v  = 0 $ in $(0,\infty) \times \R$. The middle term is zero as $v$ is smooth and vanishes  in $\{ (x,t) \in (0,\infty) \times \R : t < x \}$. 
For the last term, we notice that {it} is zero as either $\partial_t v(0,t) = \partial_t  ( v(0,t) ) $ or $\partial_x v(0,t) $ {vanishes} due to smoothness and $\partial^\mu v|_{x=0} =0$. Thus, $E'(t) = 0$. That is; the energy is conserved and we deduce $E(t) = E(0) = 0$. 
It is easy to check that this happens if and only if $v$ is constant in $\{ (x,t)  \in (0,\infty) \times \R  :  x< t \}   $. Further, this constant is zero by the smoothness and $v(t,t)= 0$. 
In summary, $v$ is a smooth function that vanishes in $\{ (x,t) \in (0,\infty) \times \R : t < x \}$ and $\{ (x,t) \in (0,\infty) \times \R : x < t \}$. Therefore, it must be zero in  $(0,\infty) \times \R$. 
That is,  $u_1 = u_2$ in  $(0,\infty) \times \R$. 

As shown above, the claim follows if $v = u_1 - u_2 $ is smooth in $[0,\infty) \times \R$. 
Let us show that this indeed holds. 
Since $v  = u_1 - u_2  \in \mathcal{D}_\Sigma' (\R^2)  - \mathcal{D}_\Sigma' (\R^2)   = \mathcal{D}_\Sigma' (\R^2)  $, we have a well-defined boundary value 
\[
\partial_x^\mu v|_{x=0} = \partial_x^\mu u_1|_{x=0} - \partial_x^\mu u_2|_{x=0} = 0
\]
which is obviously smooth.   
It suffices to show that $v \in C^\infty( (0,\infty) \times \R)$ (see Lemma \ref{ooiiqweh} below). 
%Indeed, the boundary value $\partial_x^\mu  v|_{x=0} = 0 $ is then just the classical  trace $ \lim_{\epsilon \to 0+}  \partial_x^\mu v(\epsilon,t)$ of the smooth $\partial_x^\mu  v$ in $(0,\infty) \times \R$.  below.   Further, since the derivative  $\partial_x^\mu$ is elliptic microlocally near $\Sigma \supset WF(v) $ in $\R^2$, also the element $v$ extends smoothly to the boundary.  
%Hence, we have the smooth extension of $\partial_x^\mu  v$ from the interior $(0,\infty) \times \R)$ to the boundary. 

%from the interior $(0,\infty) \times \R$  
%Notice that the boundary value $\partial_x^\mu  v|_{x=0} = 0$ 

%\footnote{J: Mikäs tässä nyt olikaan perustelu? Missaanko jotakin triviaalia.}. 
For the rest of the proof, we consider $v$ as a distribution in $\mathcal{D}'( (0,\infty) \times \R)$. 
We split $v$ microlocally into $v = v^\rightarrow + v^\leftarrow  $, where 
\begin{align}
&WF( v^\rightarrow ) \subset \Sigma^\rightarrow := \{ (x,t,-\xi,\xi)   : \xi \in \dot \R \}, \\ 
&WF( v^\leftarrow ) \subset \Sigma^\leftarrow := \{ (x,t,\xi,\xi) : \xi \in \dot \R \}. 
\end{align}
By the propagation of singularities, these wave fronts are invariant in the canonical symplectic flows in $\Sigma^{\leftarrow}$ and $\Sigma^{\leftarrow}$ inherited from $\Sigma$. 
Recall that $v=0$ in $Q:= \{ (x,t) \in (0,\infty) \times \R : t < x \} $. 
As the wave front sets of $v^\rightarrow$ and $ v^\leftarrow$ do not overlap, 
its must be that $v^\rightarrow$ and $ v^\leftarrow$ are smooth in $Q$. 
As every orbit (null bicharacteristic) in $\Sigma^{\leftarrow}$ reaches $Q \times \dot \R^2$, the component $v^\leftarrow$ is smooth in the whole $(0,\infty) \times \R$ by the propagation of singularities. 
Using the fact that $\square_A = \square =  \frac{1}{4} \partial_{z_+ } \partial_{z_-}$ in $(0,\delta) \times \R$ in the coordinates $z_\pm = t \pm x$, one can show that $v^\rightarrow(x,t) \equiv F(t-x)$, $F \in \mathcal{D}'(\R)$ is the general form of a wave in $(0,\delta) \times \R$ propagating in $\Sigma^\rightarrow$.
Above, $F(t-x)$ stands for $\phi \mapsto \int  F (s)   \phi(x,s+x)  ds dx$, $  \phi \in C^\infty_c $   in the sense of distributions. 
We have 
\[
 0 = \partial_x^\mu v|_{x= 0} = \partial_x^\mu \underbrace{v^\rightarrow}_{=F(t-x)} |_{x= 0}+ \underbrace{v^\leftarrow}_{\in C^\infty} |_{x= 0} \equiv (-1)^\mu  F^{(\mu)} \mod C^\infty.
\]
Hence $v^\rightarrow \in C^\infty( (0,\delta) \times \R)$ and therefore $v^\rightarrow \in C^\infty( (0,\infty) \times \R)$ by the propagation of singularities. 
In conclusion, $v = v^\rightarrow + v^\leftarrow \in C^\infty( (0,\infty) \times \R)+C^\infty( (0,\infty) \times \R) \subset  C^\infty( (0,\infty) \times \R)$. 
This finishes the proof. 
\end{proof}

\begin{lemma}\label{ooiiqweh}
Let $v \in \mathcal{D}'_\Sigma(\R^2)$ be a solution to (\ref{ibvp1})-(\ref{ibvp4}) with $f\in C^\infty(\R)$. Moreover, assume that $v$ is  smooth in  $  (0,\infty) \times \R $.
Then, the classical trace  $\lim_{\epsilon \to 0+}  \partial_{x,t}^\alpha v(\epsilon,t)$ exists pointwise for every $t \in \R$ and $\alpha \in \N^2$. Moreover, it is smooth in $t$ and coincides with 
$ \partial_{x,t}^\alpha  v|_{x=0}   $. 
In particular, $ v \in C^\infty  ( [0,\infty) \times \R )$. 

\end{lemma}
\begin{proof}
Let us deduce the claim in the case  $\alpha = 0$. The general claim follows by replacing $v$ with $\partial_{x,t}^\alpha v$. 
Since $v$  is continuous in $(0,\infty) \times \R$, the function $v_\epsilon(x,t) := v(x+\epsilon,t) $ is well defined pointwise on $(-\epsilon,\infty) \times \R$. Moreover, $v_\epsilon|_{x=0} (t) = v(\epsilon,t)$ pointwise. 
For $0 < \epsilon < \delta $, 
$v_\epsilon $ obeys the wave equation 
$\square v_\epsilon  = 0$ in $(-\epsilon, \delta- \epsilon) \times \R $. Moreover, the wave has smooth Cauchy data on the surface $x=0$. Thus, 
\[
v_\epsilon(x,t) = F_\epsilon( x-t) + G_\epsilon( x+t)
\]
near the Cauchy surface 
for some smooth $F_\epsilon,G_\epsilon$ determined by the data. Moreover, 
\[
v_{\epsilon}(x,t) =v_{\epsilon'} (x+ \epsilon-\epsilon',t) 
\]
so we may write 
\begin{align}
F_\epsilon(  r ) = F (  r + \epsilon ) 
\quad \quad 
G_\epsilon(  r ) = G (  r + \epsilon ) 
\end{align}
for some $F,G \in C^\infty(\R)  $. 
Hence, the pointwise limit 
\[
\lim_{\epsilon \to 0+} ( v_\epsilon|_{x=0} (t) )= \lim_{\epsilon \to 0+} v(\epsilon,t) = F(-t) + G(t)
\]
from the right  
exists and is smooth. The convergence is in fact uniform for $t $ in a compact set. 
It  is easy to check that  $v_\epsilon$
converges in $\mathcal{D}'_\Sigma(\R^2)$
to $v$. 
 As the boundary restriction is continuous on  $\mathcal{D}'_\Sigma (\R^2)$, the limit $ \lim_{\epsilon \to 0} ( v_\epsilon|_{x=0} ) $ in $\mathcal{D}'(\R) $
exists and equals to $v|_{x=0}$. 
This is the weak limit. 
Recall that also the pointwise limit exists. Moreover, the convergence to it is uniform in compact sets. Hence, the pointwise limit matches with the weak limit. 
\end{proof}

Thanks to the propositions above, we can now define the following continuous operators:

\begin{definition}[Dirichlet-to-Neumann, Neumann-to-Dirichlet operators]\label{DN-def}
Let $$\mathcal{D}'(\R)_{+} := \{ f \in \mathcal{D}'(\R) : \text{supp} (f) \subset (t_0,\infty), \ \text{for some} \ t_0 \in \R \}.$$
We define
\[
\DN  :  \mathcal{D}'(\R)_{+} \to  \mathcal{D}'(\R)_{+}, \quad  
 \DN (f) := \partial_x u |_{x=0},
 \]
  where  $u \in \mathcal{D}_\Sigma'(\R^2)$ satisfies the conditions (\ref{ibvp1}-\ref{ibvp4})  for $t_0 < \inf \text{supp} (f)$, $\mu=0$ and $f$ as the Dirichlet boundary value.  
Analogously, we define 
\[
\ND :  \mathcal{D}'(\R)_{+} \to  \mathcal{D}'(\R)_{+}, \quad  
\ND (f) :=  u |_{x=0}, 
\]
where  $u  \in \mathcal{D}_\Sigma'(\R^2)$ satisfies  the conditions (\ref{ibvp1}-\ref{ibvp4})  for $t_0 < \inf \text{supp} (f)$, $\mu=1$ and $f$ as the Neumann boundary value. 
We call $\Lambda_{DN}$ and $\Lambda_{ND}$ the Dirichlet-to-Neumann (DN) map  and Neumann-to-Dirichlet (ND) map, respectively. 
These maps are continuous on $\mathcal{D}'(\R)_{t_0} := \{ f \in \mathcal{D}'(\R) : \text{supp} (f) \subset (t_0,\infty) \}$ for a fixed $t_0 \in \R$. 
\end{definition}

\section{Derivation of Theorem \ref{themain}}\label{derivation-section}

Theorem \ref{themain} is proven in this section. There are several preliminary steps before the actual proof. 
Namely, we shall construct the following part of the diagram \eqref{tree1}:
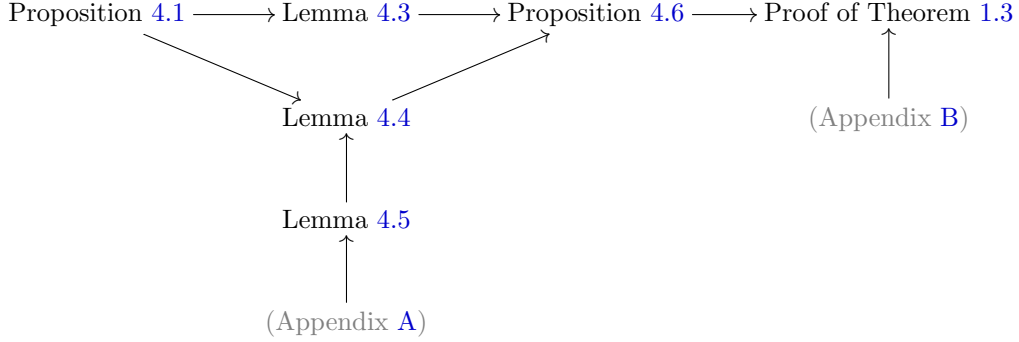
\[
\xymatrix{
\text{Proposition \ref{prop:integration-by-parts}}  \ar[r] \ar[dr] & \text{Lemma \ref{weaklemma}} \ar[r] & \text{Proposition \ref{L2-con}}  \ar[r] & \text{Proof of Theorem \ref{themain}} \\
    & \text{Lemma \ref{E-lemma}}  \ar[ur] & & \ar[u] 
    {\color{gray}\text{(Appendix \ref{sondhi_app})}}
    \\
    & \text{Lemma \ref{K-lemma}} \ar[u] \\
    & \ar[u] 
    {\color{gray}\text{(Appendix \ref{energy-estimate-sec})}}
    }
\]
%Lemma \ref{K-lemma} is based on the energy estimates computed in Section \ref{energy-estimate-sec}. Appendix \ref{sondhi_app} is applied in the final proof. 

Throughout this section, we consider the following:
Let $\chi:\R\to\R$ be a smooth function  such that $\text{supp}(\chi) \subset (0,\infty) $ and $\chi |_{[\delta,\infty)}=1$ for some $\delta >0$. Denote by $\square_A $ the second order differential operator 
\begin{equation}
  \label{def-wave-op}
 \square_A v(t,x) :=   \partial_t^2v(t,x)   - \frac{1}{A(x)} \partial_x ( A(x) \partial_x v(x,t)).
\end{equation}
where $A \in C^\infty(\R)$ is admissible (cf. Definition~\ref{def:admissible}). 
Given $\omega\in \mathcal{S}'(\R)$, we let $u=u_\omega \in \mathcal{D}_\Sigma' (\R^2) \subset \mathcal{D}' ((0,\infty) \times \R)$ solve
\begin{align}
 & \square_A u(x,t)  = 0,  && \text{for} \quad (x,t) \in    (0,\infty) \times \R ,\label{eq:wave-u}\\
 & \partial_x u |_{x=0} (t)  = \chi (t) \omega (t), &&\text{for} \quad  t \in \R,\label{eq:neumann-u}\\
 & u (x,t)= 0, &&\text{for} \quad (x,t) \in (0,\infty) \times ( - \infty, 0) .\label{eq:initial-u}
\end{align} 
As shown in Section \ref{IBVP}, such a solution exists and is unique, when considered as an element in $\mathcal{D}' ((0,\infty) \times \R)$. 
For $\phi \in C^\infty_c(\R)$, let $\tilde{u} = \tilde{u}_\phi \in C^\infty ( [0,\infty) \times \R)$ be a solution to the time-reversed initial-boundary value problem
\begin{align}
 & \square_A \tilde{u}(x,t)  = 0,  && \text{for} \quad (x,t) \in    (0,\infty) \times \R ,\label{tr:eq:wave-u}\\
 & \partial_x \tilde{u} |_{x=0} (t)  =  \phi(t), &&\text{for} \quad  t \in \R,\label{tr:eq:neumann-u}\\
 & \tilde{u} (x,t)= 0, &&\text{for} \quad (x,t) \in (0,\infty) \times ( t_0 , \infty ) .\label{tr:eq:initial-u}
\end{align} 
where $t_0 > \max \text{supp}(\phi) $. 
Existence and uniqueness in $(0,\infty) \times \R$ for $\tilde{u}$ follows from the theory developed for the forwards propagating solution $u$. 
One checks that the relation between them is 
\begin{equation}\label{rev-fut}
\tilde{u}_\phi (x,t) = u_{   \psi } (x,t_0-t), 
\end{equation}
where $\psi (t) := \phi(t_0 - t)$. 
We define the time-reversed Neumann-to-Dirichlet map $\tilde\Lambda_{ND} : C^\infty_c (\R) \to  C^\infty( \R)$ by $  \tilde\Lambda_{ND} (\phi) = \tilde{u}_{ \phi }|_{x=0} $.  
The following proposition shall be applied. 

\begin{proposition}\label{prop:integration-by-parts}
Let $A \in C^\infty(\R)$ be admissible. 
Let $f \in \mathcal{D}'(\R)$ be supported in $(0,\infty)$. 
Then, 
%{for every element $\omega \in \mathcal{S}'(\R)$, we have}
  \begin{equation}
  \langle \Lambda_{ND} (f), \phi \rangle =    \langle f , \tilde\Lambda_{ND} ( \phi )  \rangle, 
  \end{equation}
   {for all   $\phi \in C_c^\infty(\R)$.}
\end{proposition}
\begin{proof}
Let $\psi  \in C_c^\infty(\R)$ be a mollifier and denote $\psi_\varepsilon(t) = \frac{1}{\varepsilon} \psi ( \varepsilon t)$. 
For  sufficiently small values of $\varepsilon >0 $, the smooth function $f_\varepsilon := \psi_\varepsilon *  f $ is supported in $(0,\infty)$. 
Let us denote by $u_\varepsilon  $ the smooth wave $u_{f_\varepsilon}$. 
A straightforward computation gives 
\begin{align}
\langle  \ND (f_\varepsilon), \phi \rangle &= \int_{\R} u_{\varepsilon} (0,t)  \phi (t) dt =  \int_{\R} u_{\varepsilon} (0,t) A(0) \partial_x \tilde{u}_{\phi} (0,t)   dt \\
& {= - \int_{\R} \int_0^\infty   \partial_x [   u_{\varepsilon} (x,t)  A(x)  \partial_x \tilde{u}_{\phi} (x,t)   ]  dx dt 
+
 \int_{\R} \lim_{x \to \infty}  \partial_x [   u_{\varepsilon} (x,t)  A(x)  \partial_x \tilde{u}_{\phi} (x,t)   ]   dt} \\
&= - \int_{\R} \int_0^\infty   \partial_x [   u_{\varepsilon} (x,t)  A(x)  \partial_x \tilde{u}_{\phi} (x,t)   ]  dx dt + { \int_{\R}  0 dt } \label{roqw13} \\
&= -\int_{\R}\int_0^\infty  \bigg[     ( \partial_x u_{\varepsilon} (x,t)   ) A(x)  \partial_x \tilde{u}_{\phi} (x,t)   +  u_{\varepsilon} (x,t)  \partial_x (A(x) \partial_x \tilde{u}_{\phi} (x,t) )  \bigg] dx dt. 
\end{align}
{where the  identity \eqref{roqw13} is due to finite speed of propagation. Indeed, the wave 
$u_{\varepsilon} (x,t)$ has null Cauchy data on $(0,\infty) \times \{0\}$  and therefore  vanishes in the wedge $\{(x,t) \in (0,\infty) \times \R : |t| < x\}$. Further, }
%
%\begin{align}
%&= -%\int_{\R}\int_0^\infty  %\bigg[     ( \partial_x %u_{\varepsilon} (x,t)   %) A(x)  \partial_x %\tilde{u}_{\phi} (x,t)   %+  u_{\varepsilon} (x,t) % \partial_x (A(x) %\partial_x %\tilde{u}_{\phi} (x,t) ) % \bigg] dx dt. 
%\end{align}
integration by parts in $x$ yields 
\begin{align}
&\langle  \ND (f_\varepsilon), \phi \rangle  = \int_{\R} \int_0^\infty  \bigg[    \tilde{u}_{\phi} (x,t)    \partial_x ( A(x) \partial_x u_{\varepsilon} (x,t)   )   +( A(x) \partial_x u_{\varepsilon} (x,t) )   \partial_x \tilde{u}_{\phi} (x,t)  \bigg] dx dt.  \\
 &+ \int_{\R}   \bigg[  \tilde{u}_{\phi} (0,t)      \partial_x u_{\varepsilon} (0,t)      + u_{\varepsilon} (0,t)    \partial_x \tilde{u}_{\phi} (0,t)  \bigg] dt \\
& = \int_{\R} \int_0^\infty   \partial_x [  A(x) u_{\varepsilon} (x,t)   \partial_x \tilde{u}_{\phi} (x,t)  ]dx dt.  
 + \int_{\R}   \bigg[  \tilde{u}_{\phi} (0,t)      \partial_x u_{\varepsilon} (0,t)     + u_{\varepsilon} (0,t)    \partial_x \tilde{u}_{\phi} (0,t)  \bigg] dt \\
 & = -  \int_{\R}  \underbrace{A(0)}_{=1} u_{\varepsilon} (0,t)   \partial_x \tilde{u}_{\phi} (0,t)   dt.  
 +\int_{\R} \bigg[   \tilde{u}_{\phi} (0,t)      \partial_x u_{\varepsilon} (0,t)     + u_{\varepsilon} (0,t)    \partial_x \tilde{u}_{\phi} (0,t)  \bigg] dt \\
& =  \int_{\R}    \tilde{u}_{\phi} (0,t)      \partial_x u_{\varepsilon} (0,t)   dt =  \int_{\R}     \tilde{u}_{\phi} (0,t)     f_\varepsilon(t)   dt  = \langle f_\varepsilon, \tilde\Lambda_{ND} (\phi) \rangle.
\end{align}
By taking the limit $\varepsilon \to 0$ and applying continuity of $\ND$ on $\mathcal{D}'(\R)_{0} := \{ f \in \mathcal{D}'(\R) : \text{supp} (f) \subset (0,\infty) \}$, we conclude
\[
\langle  \ND ( f ), \phi \rangle  =  \langle f , \tilde\Lambda_{ND} (\phi) \rangle.
\]
\end{proof}

\begin{definition}
  Denote by $\tau_s$ the time-shift $\tau_s(t) = t+s$. Let $T>0$ and let $A$ be admissible. 
 % { Given $\omega \in \mathcal{S}'(\R)$, }
  The correlation operator $C_T  : {\mathcal{S}'(\R ) \times  }C^\infty_c ( \R) \to \mathcal{D}'(\R)$ is defined by 
  \begin{equation}\label{eq:correlation-function-def}
    \langle C_T (\omega, \phi ),   \psi  \rangle := \frac1T\int_0^T \langle \tau_s^* \ND (\chi \omega ), \phi  \rangle \langle \tau_s^* \omega,  \psi  \rangle ds
  \end{equation}
  where $\phi, \psi \in C^\infty_c(\R)$. 
   { We shall  omit the parameter  $\omega$ and denote $C_T (\phi):= C_T(\omega,\phi)  \in \mathcal{D}'(\R)$ whenever there is no danger of confusion. 
   }
 {The correlation operator is inherently stochastic and may be regarded as an operator-valued random variable. Alternatively, via Schwartz’s kernel theorem, it can be viewed as a $\mathcal{D}'(\R^2)$-valued random variable.   }
 \end{definition}

We would like to study the $L^2(\mathbb{P})$-limit of the correlation operator as $T \to \infty$. 
The first step in that is to compute the ``weak'' limit. 

 \begin{lemma}\label{weaklemma}
 Let $\phi , \psi \in C^\infty_c(\R)$. Then 
\[ \lim_{T \to \infty} \E \langle     C_T ( \phi) , \psi \rangle = \langle  \tilde\Lambda_{ND} (\phi  ) ,   \psi   \rangle.
\]
 \end{lemma}
 
 \begin{proof}
Denote
  \begin{equation}
    X_s := \langle \tau_s^* \ND (\chi \omega ), \phi  \rangle, \qquad Y_s := \langle \tau_s^* \omega, \psi \rangle.
  \end{equation}
Proposition \ref{prop:integration-by-parts} implies that 
$
X_s =  \langle  \omega , \tilde\phi \rangle$ where $ \tilde\phi := \chi  \tilde\Lambda_{ND} (\tau_{-s}^* \phi ) \in C^\infty_c(\R) \subset \mathcal{S}(\R). $
Similarly, $Y_s = \langle \omega, \tilde\psi \rangle$ for $\tilde\psi:= \tau_{-s}^* \psi \in C^\infty_c(\R) \subset \mathcal{S}(\R) $. 
By the isometry \eqref{itolike},
\[
 \E (X_s Y_s) = \langle \tilde\phi , \tilde\psi \rangle =\langle  \chi  \tilde\Lambda_{ND} ( \tau_{-s}^*\phi  ),   \tau_{-s}^*\psi  \rangle .    
\]
It is easy to check that  $\tilde\Lambda_{ND} ( \tau_{-s}^* \phi ) =   \tau_{-s}^*\tilde\Lambda_{ND} (\phi)$. Substituting this in the previous equation leads to 
\begin{align}
 \E (X_s Y_s) =\langle  \chi  \tau_{-s}^* \tilde\Lambda_{ND} (\phi  ) ,  \tau_{-s}^*   \psi  \rangle   = \langle  ( \tau_{s}^* \chi  )   \tilde\Lambda_{ND} (\phi  ) ,   \psi  \rangle. 
 \label{k10}
 \end{align}
 Moreover, by Hölder's inequality and  the isometry \eqref{itolike}, 
 \[
  \E |X_s Y_s| \leq  \| X_s \|_{L^2(\mathbb{P})}  \| Y_s \|_{L^2(\mathbb{P})} =   \| \tilde\phi \|_{L^2(\R)}  \| \tilde\psi  \|_{L^2(\R)} < \infty. 
 \]
 Hence, $\frac{1}{T} \int_0^T  \E  |X_s Y_s|  ds < \infty.$
By Fubini's theorem,
 \[
\E \langle C_T ( \phi) ,\psi \rangle = \E  \frac{1}{T}  \int_0^T   X_s Y_s ds  = \frac{1}{T}  \int_0^T    \E (X_s Y_s)  ds .
\]
Applying \eqref{k10} we obtain 
\begin{align}
\E \langle C_T ( \phi) ,\psi \rangle =  \frac{1}{T}  \int_0^T    \langle  ( \tau_{s}^* \chi  )   \tilde\Lambda_{ND} (\phi  ) ,   \psi  \rangle ds  =     \langle   \chi_T  \tilde\Lambda_{ND} (\phi  ) ,   \psi   \rangle ,\label{wooq}
\end{align}
where 
\[
\chi_T(t) :=   \frac{1}{T}  \int_0^T   \chi  (t+s)  ds  .
\]
We have 
\begin{align}
&|1- \chi_T(t)  | 
= \left| \frac{1}{T} \int_0^T (1-\chi(t+s))  ds   \right|
=
 \left| \frac{1}{T} \int_0^{\min\{T,\delta-t\}} (1-\chi(t+s))  ds   \right|
 \\
& \leq 
 \frac{1}{T} \int_0^{\min\{T,\delta-t\}}  \left| (1-\chi(t+s))    \right|ds
 \leq
  \frac{1}{T} \int_0^{\delta-t}  (1+\|\chi\|_{L^\infty(\R)}) ds 
  =  
  \frac{ (1+\|\chi\|_{L^\infty(\R)}) (\delta - t)}{T},
  %\to 0,
  %\quad \text{as} \quad T \to 0,
\end{align}
%as $T \to 0$, 
which implies that $\chi_T(t) \to 1$ pointwise as $T$ grows.  Moreover, 
\[
| \chi_T(t) |  \leq  \left| \frac{1}{T} \int_0^T \chi(t+s)   ds  \right|  \leq   \frac{1}{T} \int_0^T \left| \chi(t+s)     \right|ds
\leq   \frac{1}{T} \int_0^T   \| \chi \|_{L^\infty(\R)} ds  = \| \chi \|_{L^\infty(\R)}.
\]
Consequently,
$
 \chi_T  \tilde\Lambda_{ND} (\phi  )   \psi   \to \tilde\Lambda_{ND} (\phi  )   \psi  
$ 
pointwise and $| \chi_T  \tilde\Lambda_{ND} (\phi  )   \psi |$ is dominated by the integrable function $ \| \chi \|_{L^\infty(\R)} | \tilde\Lambda_{ND} (\phi  )   \psi| $.
By the dominated convergence theorem,  
\[
\lim_{T \to \infty}\langle   \chi_T  \tilde\Lambda_{ND} (\phi  ) ,   \psi   \rangle
=
\lim_{T \to \infty} \int_{\R}   \chi_T(t)  \tilde\Lambda_{ND} (\phi  )(t)    \psi(t) dt
=
\int_{\R}  \tilde\Lambda_{ND} (\phi  )(t)    \psi(t) dt
=
 \langle    \tilde\Lambda_{ND} (\phi  ) ,   \psi   \rangle .
\]
Combining this with \eqref{wooq} yields the result.
%\[
% \lim_{T \to \infty} \E \langle     C_T ( \phi) %, \psi \rangle = \langle  \tilde\Lambda_{ND} %(\phi  ) ,   \psi   \rangle. 
%\]
%In conclusion, 
%\[
%\E   \langle C_T ( \phi)-    \tilde\Lambda_{ND} (\phi  ) ,   \psi   \rangle      = \E   \langle C_T ( \phi), \psi \rangle - \langle  \tilde\Lambda_{ND} (\phi  ) ,   \psi   \rangle \to 0. 
%\]
\end{proof}

The second step in computing the $L^2(\mathbb{P})$-limit of $C_T$ is the following:

\begin{lemma}\label{E-lemma}
Let  $\psi \in C^\infty_c(\R)$ and let $\phi \in C^\infty_c( \R)$ be such that $\int_{\R} \phi(t)dt  = 0$. 
Then 
\[
\lim_{T \to \infty}  \E \left(  \langle C_T (\phi ),   \psi  \rangle^2 \right)= \langle \tilde\Lambda (\phi) ,\psi \rangle^2 .
\] 
\end{lemma}  

\begin{proof}
Denote
  \begin{equation}
    X_s := \langle \tau_s^* \ND (\chi \omega ), \phi  \rangle, \qquad Y_s := \langle \tau_s^* \omega, \psi \rangle.
  \end{equation}
It follows from the definition of white noise measure that $  \langle \omega, \varphi \rangle $  is a Gaussian random variable with zero mean for every $\varphi \in \mathcal{S}$. 
Applying Proposition \ref{prop:integration-by-parts}, we get that 
$
X_s =  \langle  \omega , \tilde\phi_s \rangle$ where $ \tilde\phi_s := \chi  \tilde\Lambda_{ND} (\phi \circ \tau_{-s} ) \in C^\infty_c(\R) \subset \mathcal{S}(\R). $
Hence, $X_s$ is a Gaussian random variable with zero mean. 
Similarly, $Y_s = \langle \omega, \tilde\psi_s \rangle$ for $\tilde\psi_s:= \psi \circ \tau_{-s} \in C^\infty_c(\R) \subset \mathcal{S}(\R) $ so $Y_s$ is a Gaussian random variable with zero mean. 
According to Isserli's formula \cite{isserlis18_formul_produc_momen_coeff_any}, 
  \begin{align}\label{1eq:isserlis}
    \E[&X_sY_sX_{r}Y_{r}] = \E[X_sY_s] \E[X_{r}Y_{r}] 
       + \E[X_sX_{r}] \E[Y_sY_{r}] + \E[X_sY_{r}] \E[X_{r} Y_s]. 
  \end{align}
By the isometry \eqref{itolike}, 
  \begin{align}\label{isse2}
    \E[X_sY_sX_{r}Y_{r}] &= \langle  \tilde\phi_s,  \tilde\psi_s \rangle  \langle  \tilde\phi_{r},  \tilde\psi_{r}  \rangle+\langle  \tilde\phi_s,  \tilde\phi_{r} \rangle  \langle  \tilde\psi_{s},  \tilde\psi_{r}  \rangle+\langle  \tilde\phi_s,  \tilde\psi_{r} \rangle  \langle  \tilde\phi_{r},  \tilde\psi_{s}  \rangle. 
    \\
    &= \langle \tilde\Lambda_{ND} ( \tau_{-s}^* \phi  ),  \chi ( \tau_{-s}^* \psi ) \rangle \langle \tilde\Lambda_{ND} (\tau_{-r} ^* \phi ),  \chi (\tau_{-r}^* \psi )  \rangle \\
    &+\langle  \tilde\Lambda_{ND} (\tau_{-s}^* \phi ),  \chi^2  \tilde\Lambda_{ND} (\tau_{-r}^* \phi  )  \rangle \langle  \tau_{-s}^* \psi  ,   \tau_{-r}^* \psi  \rangle  \\
    &+ \langle   \tilde\Lambda_{ND} (\tau_{-s}^*\phi ),  \chi( \tau_{-r}^*\psi  ) \rangle \langle \tilde\Lambda_{ND} ( \tau_{-r}^*\phi  ),  \chi(  \tau_{-s}^*\psi ) \rangle. 
  \end{align}
One checks that $  \tilde\Lambda_{ND} ( \tau_{-s}^* \phi  ) =  \tau_{-s}^* \tilde\Lambda_{ND} (\phi)  $. Similarly, $  \tilde\Lambda_{ND} ( \tau_{-r}^* \phi  ) = \tau_{-r} ^*\tilde\Lambda_{ND} (\phi) $. 
Substituting this in the previous equation gives 
  \begin{align}\label{isse3}
   \E[X_sY_sX_{r}Y_{r}]  &= \langle \tilde\Lambda_{ND} (\phi  ),  ( \tau_{s}^* \chi  )   \psi  \rangle \langle \tilde\Lambda_{ND} (\phi ),   (\tau_{r}^* \chi   ) \psi  \rangle \\ 
       &+\langle  \tilde\Lambda_{ND} (\phi ),  (\tau_{{s}}^* \chi )^2   \tau_{{s-r}}^* \tilde\Lambda_{ND} (\phi )   \rangle \langle  \psi ,  \tau_{s-r}^*\psi  \rangle  \\
           &+ \langle  \tilde\Lambda_{ND} (\phi  ),  (\tau_s^* \chi ) (\tau_{s-r}^* \psi ) \rangle \langle   \tilde\Lambda_{ND} (\phi ), (\tau_r^* \chi ) ( \tau_{r-s}^* \psi  )\rangle. 
\end{align} 
Thus,
\begin{align}
\E  \langle  C_T (\phi ), \psi \rangle^2  &=  \frac{1}{T^2} \int_0^T \int_0^T  \E[X_sY_sX_{r}Y_{r}]  ds dr \\ 
&= + \frac{1}{T^2} \int_0^T \int_0^T   \langle \tilde\Lambda_{ND} (\phi  ),  (\tau_{s}^* \chi)   \psi  \rangle \langle \tilde\Lambda_{ND} (\phi ),   (\tau_{r}^*\chi   ) \psi  \rangle  ds dr\label{isse41} \\  
       &+ \frac{1}{T^2} \int_0^T \int_0^T\langle  \tilde\Lambda_{ND} (\phi ),  (\tau_{{s}}^* \chi )^2  (  \tau_{{s-r}}^*\tilde\Lambda_{ND} (\phi ) )  \rangle \langle  \psi ,  \tau_{s-r}^*\psi  \rangle  ds dr\label{isse42} \\
           &+  \frac{1}{T^2} \int_0^T \int_0^T\langle  \tilde\Lambda_{ND} (\phi  ),  (\tau_s^*\chi)  \tau_{s-r}^*\psi  \rangle \langle   \tilde\Lambda_{ND} (\phi ), ( \tau_r^*\chi )  \tau_{r-s}^*\psi  \rangle ds dr. \label{isse43}
\end{align}
The first term on the right equals $ \langle  \tilde\Lambda_{ND} (\phi  ),  \tilde\chi_T \psi  \rangle \langle  \tilde\Lambda_{ND} (\phi ), \tilde\chi_T  \psi  \rangle $, where $\tilde\chi_T (t)  = \frac{1}{T} \int^T_0 \chi (t+s) ds $.
The term converges {to}  $ \langle   \tilde\Lambda_{ND} (\phi  ),  \psi  \rangle^2$ as $ T \to \infty$. 
The second term can be written in the form
\begin{align}
&   \frac{1}{T^2} \int_0^T \int_0^{t_\psi} \langle   \tilde\Lambda_{ND} (\phi ),  (\tau_{{\tilde{s}+r}}^*\chi )^2   \tau_{{\tilde{s}}}^*\tilde\Lambda_{ND} (\phi )   \rangle \langle  \psi ,  \tau_{\tilde{s}}^* \psi  \rangle  d\tilde{s} dr ,
 \end{align}
 where $t_\psi := \text{diam}  ( \text{supp} ( \psi )   ) < \infty$ and $ \tilde{s} = s-r$. 
 We would like to show that this converges to zero as $T \to \infty$. 
Since the Cauchy-Schwarz inequality implies $| \langle  \psi ,   \tau_{\tilde{s}}^*\psi  \rangle|  \leq \| \psi \|^2_{L^2} < \infty$, the limit relies on the behaviour of  $ \langle   \tilde\Lambda_{ND} (\phi ),  (\tau_{{\tilde{s}+r}}^*\chi )^2  \tau_{{\tilde{s}}}^*\tilde\Lambda_{ND} (\phi )   \rangle $ as $r $ grows. 
It suffices to prove that 
\begin{equation}\label{19q}
|  \tilde\Lambda_{ND} (\phi)(t) | \leq C e^{-C'|t|}  , \quad \text{i.e.}  \quad |  \tilde{u}_\phi  (0,t)  | \leq C e^{-C'|t|}  , 
\end{equation}
 for some $C,C'>0$. 
Before going into that, let us study the last term \eqref{isse43}. It reads
\[
 \frac{1}{T^2} \int_0^T \int_0^T\langle  \tilde\Lambda_{ND} (\phi  ),  ( \tau_{\tilde{s}+r}^* \chi  ) \tau_{\tilde{s}}^* \psi  \rangle \langle   \tilde\Lambda_{ND} (\phi ), (\tau_r^* \chi )  \tau_{-\tilde{s}}^* \psi   \rangle d \tilde{s} dr.
\]
The supports of the test functions $ ( \tau_{\tilde{s}+r}^* \chi  ) \tau_{\tilde{s}}^* \psi  $ and  $(\tau_r^* \chi )  \tau_{-\tilde{s}}^* \psi  $  are shifted further as $\tilde{s}$ and $\tilde{r}$ grow. 
Hence, also this integral converges to $0$ provided the decay estimate \eqref{19q}. 
The estimate indeed holds, as shown in the lemma below. This finishes the proof. 
\end{proof}

\begin{lemma}\label{K-lemma}
Let $\phi \in C^\infty_c(\R)$ be such that $ \int_{\R} \phi (t) dt = 0$. Then, there are $C,C'>0$ such that the estimate \eqref{19q} holds.
\end{lemma} 

{The proof of this lemma relies on the energy estimates derived in Appendix  \ref{energy-estimate-sec}. The following proposition is a consequence of these estimates. For a proof, see p. \pageref{1decay-pro} in the appendix. 
%Namely, we need the following:

%\begin{proposition}\label{1decay-pro}
\begin{repproposition}{1decay-pro}
  Let $u = u_\phi  \in C^\infty( [0,\infty) \times \R)  $ satisfy 
  \begin{align}
    \square_A u(x,t) &= 0, &&   (x,t) \in (0,\infty) \times \R,\\
    \partial_x u (0,t) &= \phi(t), && t\in\R,\\
    u(x,t) &= 0, && (x,t) \in (0,\infty) \times (-\infty,0).
  \end{align}
  where $\phi \in C^{\infty}_c (\R)$ is supported in $(0,\infty)$. Assume furthermore that $A(x)=A_\infty$ for $x > \ell$. Then, there are $C,C'>0$ such that 
 \[
    \big|  u(0,t) -c(t)  \big| \leq C e^{-C't}, \quad \forall t \in \R, 
    \]
    where  $c(t) := - \frac{1}{A_\infty} \int_0^t \phi(s) ds$. 
\end{repproposition}
}

\begin{proof}[Proof of Lemma \ref{K-lemma}]
Recall that the time-reversed wave is related to the standard one via $\tilde{u}_\phi (x,t) = u_{\psi} (x,t_0- t)$, where $ \psi (t) = \phi(t_0- t)$ and $t_0  > \max \text{supp} (\phi) $. 
Hence, it suffices to study the decay of $ u_{\psi}$ instead. 
%We apply the energy estimates derived in Section \ref{energy-estimate-sec}. 
According to Proposition \ref{1decay-pro}, 
\[
    \abs{\tilde{u}_\phi (0,t) - c(t_0-t)}=     \abs{u_\psi (0,t_0-t) - c(t_0-t)} \leq  C e^{-C' (t_0-t)}.
\]
where $c(t) :=  -\frac{1}{A_\infty} \int_0^t \psi (s) ds = -\frac{1}{A_\infty}\int_0^t \phi (t_0- s) ds $. The function $c(t_0-t)  $ vanishes for large $|t| $ by the assumption $\int_{\R} \phi (t) dt = 0$. Hence, 
\[
    \abs{\tilde{u}_\phi (0,t) } \leq   C e^{-C' (t_0-t)}.
\]
for such  $t$. 
The decay of  $C e^{-C' (t_0-t)}$ towards $-\infty$ is exponential, whereas 
 the wave $\tilde{u}_\phi (0,t) = u_\psi (0,t_0-t) $ vanishes identically for $t > t_0$ by the initial condition. 
After updating the constants $C',C>0$, we arrive at $|\tilde{u}_\phi (0,t) | \leq   C e^{-C' |t|}$.

\end{proof}

We now put together the previous lemmas to prove the following $L^2( \mathbb{P})$ limit.

\begin{proposition}\label{L2-con}
Let $A$ be admissible. Let $\phi, \psi \in C^\infty_c(\R)$ and assume that $\int_{\R}  \phi (t) dt = 0$. 
Then,
\[
\langle C_T(\phi) , \psi \rangle  \to \langle \tilde\Lambda_{ND} (\phi ), \psi \rangle 
\]
in $L^2( \mathbb{P})$ as $T \to \infty$. 
\end{proposition}

\begin{proof}
A straightforward computation gives 
\[
\E \left( \left( \langle C_T(\phi) , \psi \rangle  - \langle \tilde\Lambda_{ND} (\phi ), \psi \rangle  \right)^2 \right)= \E \left( \langle C_T(\phi) , \psi \rangle^2   \right) - 2   \langle \tilde\Lambda_{ND} (\phi ), \psi \rangle \E  \langle C_T(\phi) , \psi \rangle   +  \langle \tilde\Lambda_{ND} (\phi ), \psi \rangle^2. 
\]
By Lemma \ref{weaklemma}, $ \E  \langle C_T(\phi) , \psi \rangle \to  \langle \tilde\Lambda_{ND} (\phi ), \psi \rangle$, as $T \to \infty$. 
Moreover,  $ \E (  \langle C_T(\phi) , \psi \rangle^2 ) \to  \langle \tilde\Lambda_{ND} (\phi ), \psi \rangle^2$ by Lemma \ref{E-lemma}. 
Thus, 
\[
 \E \left( \left( \langle C_T(\phi) , \psi \rangle  - \langle \tilde\Lambda_{ND} (\phi ), \psi \rangle  \right)^2 \right)  \to  \langle \tilde\Lambda_{ND} (\phi ), \psi \rangle^2 - 2   \langle \tilde\Lambda_{ND} (\phi ), \psi \rangle  \langle \tilde\Lambda_{ND} (\phi ), \psi \rangle+ \langle \tilde\Lambda_{ND} (\phi ), \psi \rangle^2 = 0.
\]
as $T \to \infty$. That is; 
$
\langle C_T(\phi) , \psi \rangle  \to \langle \tilde\Lambda_{ND} (\phi ), \psi \rangle 
$
in $L^2( \mathbb{P})$. 
\end{proof}

We can finally derive the main result:
\begin{proof}[Proof of Theorem \ref{themain}]
Let $A_\nu$, $\nu=1,2$ be admissible and denote by 
 $C_{T\nu} $, $\Lambda_{\nu} $ and $\tilde\Lambda_{\nu} $ for $\nu=1,2$ the associated correlation operators, Neumann-to-Dirichlet maps and time-reversed Neumann-to-Dirichlet maps, respectively.
If $A_1 = A_2$, then $\Lambda_{1} (\chi \omega) = \Lambda_{2} (\chi \omega) $ trivially for every $ \omega \in \mathcal{S}'(\R)$. 
That is; the condition $\Lambda_{1}(\chi \omega) \neq \Lambda_{2}(\chi \omega)$ for some $\omega \in \mathcal{S}'(\R)$ implies $A_1 \neq A_2$.

Assume then $A_1 \neq A_2$. We need to show that there is a measurable $U \subset \mathcal{S}'(\R)$  with $\mathbb{P}(U) = 1$ such that $\Lambda_{1}(\chi \omega) \neq \Lambda_{2}(\chi \omega)$ for every $\omega \in U$.  
Consider test functions $\phi_k, \psi \in C^\infty_c (\R)$, $k =1,2,3,\dots$ such that $\int \phi_k (x) dx = 0$. We shall fix $\phi_k$ more precisely later. 
 By Hölder's inequality, 
 \[
0\leq   \left\|  \langle C_{T\nu}(\phi_k) , \psi \rangle  - \langle \tilde\Lambda_{\nu } (\phi_k ), \psi \rangle  \right\|_{L^1(\mathbb{P})} \leq \left\|  \langle C_{T\nu}(\phi_k) , \psi \rangle  - \langle \tilde\Lambda_{ \nu} (\phi_k ), \psi \rangle  \right\|_{L^2(\mathbb{P})}^2 .
 \]
The dominant on the right hand side converges to $0$ as $T \to \infty$ by Proposition \ref{L2-con}. 
Consequently, 
\[
\langle C_{T\nu} (\phi_k) , \psi \rangle  \to  \langle \tilde\Lambda_{ \nu} (\phi_k ), \psi \rangle 
\]
in $L^1 ( \mathbb{P})$ as $T \to \infty$. 
In particular, 
\[
\langle (C_{T1} - C_{T2}) (\phi_k) , \psi \rangle  \to  \langle \tilde\Lambda_{ 1} (\phi_k )-  \tilde\Lambda_{ 2} (\phi_k ), \psi \rangle 
\]
in $L^1 ( \mathbb{P})$ as $T \to \infty$. 
Recall that any sequence converging in $L^1 $ admits a sub-sequence that converges pointwise almost surely. 
Hence, there are increasing $T_{1j} $, $j=1,2,3,\dots$, such that  $\langle C_{T_{1j}1}(\phi_1)-C_{T_{1j}2}(\phi_1)  , \psi \rangle$ converges to $ \langle \tilde\Lambda_{1} (\phi_1 )-\tilde\Lambda_{2} (\phi_1 ), \psi \rangle$ pointwise almost surely. 
That is; there is a measurable set $U_1 \subset \mathcal{S}'(\R)$ such that $\mathbb{P} (U_1) = 1$ and 
\begin{equation}\label{eqp:991221}
\lim_{j \to \infty} \langle C_{T_{1j}1}(\phi_1)-C_{T_{1j}2}(\phi_1)  , \psi \rangle =  \langle \tilde\Lambda_{1} (\phi_1 )-\tilde\Lambda_{2} (\phi_1 ), \psi \rangle
\end{equation}
at every $\omega \in U_1$. We can use the same argument to find an increasing subsequence $T_{2j}$ of $T_{1j}$ and measurable $U_2  \subset \mathcal{S}'(\R)$ such that $\mathbb{P} (U_2) = 1$ and 
\begin{equation}\label{eqp:991222}
\lim_{j \to \infty} \langle C_{T_{2j}1}(\phi_2)-C_{T_{2j}2}(\phi_2)  , \psi \rangle =  \langle \tilde\Lambda_{1} (\phi_2 )-\tilde\Lambda_{2} (\phi_2 ), \psi \rangle
\end{equation} 
at every $\omega \in U_2$. 
Moreover, we may assume $U_2 \subset U_1$ by 
taking intersection with $U_1$. 
Repeating this construction for every $k=1,2,3,\dots$ we obtain 
subsequences $(T_{kj})$, $k=1,2,3\dots$ and nested 
$U_1  \supset U_{2} \supset  \cdots $ such that $\mathbb{P} (U_k) = 1$ and 
\begin{equation}\label{eqp:991223}
\lim_{j \to \infty} \langle C_{T_{kj}1}(\phi_k)-C_{T_{kj}2}(\phi_k)  , \psi \rangle =  \langle \tilde\Lambda_{1} (\phi_k )-\tilde\Lambda_{2} (\phi_k ), \psi \rangle
\end{equation} 
at every $\omega \in U_k$.  
Define $T_j := T_{jj}$ and $U := \cap_{k} U_k$.
Then, $\mathbb{P} (U) =\lim_{k\to \infty} \mathbb{P} (U_k) = 1$ and 
\begin{equation}\label{eqp:99122}
\lim_{j \to \infty} \langle C_{T_{j}1}(\phi_k)-C_{T_{j}2}(\phi_k)  , \psi \rangle =  \langle \tilde\Lambda_{1} (\phi_k )-\tilde\Lambda_{2} (\phi_k ), \psi \rangle
\end{equation} 
for every $k=1,2,3,\dots$ and $\omega \in U$.

Fix $\varphi \in C^\infty_c(\R)$ and  
define the sequence $\phi_k \in C^\infty_c(\R)$, $k=1,2,3,\dots$ 
by $\phi_k = \varphi - \varphi (x + k)$. 
Analogously to $\Lambda_\nu $ being continuous on $\mathcal{D}_{t_0}' (\R)$, the operator $ \tilde\Lambda_{\nu}$
is continuous on the time-reversed domain $\tilde{\mathcal{D}}_{t_0}' (\R) :=  \{ f \in \mathcal{D}'(\R) :   \text{supp} (f) \subset (-\infty, t_0)  \} $ and we have for every $\psi \in C_c^\infty(\R)$ that 
\[
 \lim_{k \to \infty }   \langle \tilde\Lambda_{1} (\phi_k )-\tilde\Lambda_{2} (\phi_k ), \psi \rangle =   \langle \tilde\Lambda_{1} (\varphi )-\tilde\Lambda_{2} (\varphi ), \psi \rangle
\]
at every  $\omega \in U$. Moreover, the left hand side is 
\[
\lim_{k \to \infty }  \lim_{j \to \infty } \langle C_{T_{j}1}(\phi_k)-C_{T_{j}2}(\phi_k)  , \psi \rangle  
\]
by \eqref{eqp:99122}.  
Thus, 
\[
\lim_{k \to \infty }  \lim_{j \to \infty } \langle C_{T_{j}1}(\phi_k)-C_{T_{j}2}(\phi_k)  , \psi \rangle  =   \langle \tilde\Lambda_{1} (\varphi )-\tilde\Lambda_{2} (\varphi ), \psi \rangle
\]
The test functions $\varphi , \psi \in C^\infty_c( \R)$ above can be chosen freely. 
Since  $A_1 \neq A_2$, it follows from Theorem \ref{sondhi_theorem} and the continuity of $\tilde\Lambda_{1},\tilde\Lambda_{2}$ that $\tilde\Lambda_{1}( \varphi ) \neq \tilde\Lambda_{2} (\varphi  )$ for some $\varphi \in C^\infty_c (\R)$. 
Hence, the right hand side of the limit above is non-zero for some $\psi \in C_c^\infty(\R)$. 
Consequently,  $C_{T_j1}(\phi_k)-C_{T_j2}(\phi_k) \neq 0 $ for some $j$ and $k$ which further implies that $\Lambda_{1} (\chi \omega)  \neq \Lambda_{2} (\chi \omega) $ by the definition of the correlation function. 
In conclusion, there is a measurable $U \subset \mathcal{S}'(\R)$ with $\mathbb{P} (U) = 1$ such that $\Lambda_{1} (\chi \omega)  \neq \Lambda_{2} (\chi \omega) $ for every $\omega \in U$. 
This finishes the proof. 
\end{proof}

\appendix
\appendixpage

\section{Energy estimates}\label{energy-estimate-sec}

In this section, we prove that $\ND ( \phi )$ decays fast enough for $\phi \in C^\infty_c(\R)$ supported in $(0,\infty)$. 
By \eqref{rev-fut}, $\tilde\Lambda_{ND} ( \phi)$ has the same decay for arbitrary $\phi \in C^\infty_c(\R)$. 
We follow the proof idea from \cite{arnold22_expon_time_decay_one_dimen}, which in our case goes as follows.

Let $A$ be admissible. Assume that $A(x)=A_\infty$ outside of $(0,\ell)$ in (\ref{def-wave-op}). Let $\phi \in C^\infty_c(\R)$ have support in $x >0$.  We will give a precise definition of two strictly positive mapping $g_1,g_2 \colon [0,\infty) \to (0,\infty)$ later. For $\tau\in\R$, define the energy located in that interval at time $t=\tau$ by
\begin{equation}
  \label{eq:energy-def}
  \mathscr{E}_{\phi,g_1,g_2}(\tau) := 
  \frac{1}{2} \int_0^\ell \left[ g_1(x)\big(\partial_t u(x,\tau) + \partial_x u(x,\tau)\big)^2 + g_2(x)\big(\partial_t u(x,\tau) - \partial_x u(x,\tau)\big)^2 \right] A(x) \dx
\end{equation}
where $u \in C^\infty ( [0,\infty) \times \R)$ satisfies
\begin{align}
  \square_A u &= 0, && x>0,\, t\in\R,\label{eq:energy-wave-u}\\
 \partial_x u(0,t) &= \phi(t), && t\in\R,\\
  u(x,t) &= 0, && x>0,\, t<0,
\end{align}
and $\square_A$ is as defined in (\ref{def-wave-op}). 
The Neumann-to-Dirichlet map at $\phi$ is $\ND(\phi) = u|_{x=0}$. 
Because everything is smooth and bounded, we can differentiate under the integral sign. We get
\begin{align*}
  &\mathscr{E}_{\phi,g_1,g_2}'(\tau) = \int_0^\ell \left[
    g_1 (\partial_t u + \partial_x u) (\partial_t^2 u + \partial_t\partial_x u) + g_2 (\partial_t u - \partial_x u) (\partial_t^2 u - \partial_t\partial_x u)
    \right] A \dx \\
    \end{align*}
 Applying $\square_A u = 0$ yields 
    \begin{align*}
      \mathscr{E}_{\phi,g_1,g_2}'(\tau)
      &= \int_0^\ell \Big[ g_1 (\partial_t u + \partial_x u) \Big(\partial_x^2 u + \frac{\partial_x A}{A} \partial_x u \\
      &\qquad\quad + \partial_t\partial_x u\Big)  + g_2 (\partial_t u - \partial_x u) \Big(\partial_x^2 u + \frac{\partial_x A}{A} \partial_x u - \partial_t\partial_x u \Big) \Big] A \dx.  \\
      &= \int_0^\ell \Big[ g_1 (\partial_t u + \partial_x u)(\partial_x^2 u + \partial_x \partial_t u) + g_1 (\partial_t u + \partial_x u) \frac{\partial_x A}{A} \partial_x u \\
      &\qquad\quad + g_2 (\partial_t u - \partial_x u) (\partial_x^2 u - \partial_x \partial_t u) + g_2 (\partial_t u - \partial_x u) \frac{\partial_x A}{A} \partial_x u \Big] A \dx \\
      &= \frac{1}{2} \int_0^\ell \Big[ g_1 \partial_x (\partial_t u + \partial_x u)^2 - g_2 \partial_x (\partial_t u - \partial_x u)^2 + \frac{\partial_x A}{A} \big(g_1 (\partial_t u + \partial_x u)^2 - g_2 (\partial_t u - \partial_x u)^2\big) \\
      &\qquad\qquad - \frac{\partial_x A}{A} (g_1 - g_2) \big((\partial_t u)^2 - (\partial_x u)^2 \big) \Big] A \dx \\
      &= \frac{1}{2} \int_0^\ell \Big[ g_1 \partial_x \big( A (\partial_t u + \partial_x u)^2 \big) - g_2 \partial_x \big( A (\partial_t u - \partial_x u)^2 \big) - \partial_x A (g_1 - g_2) \big( (\partial_t u)^2 - (\partial_x u)^2 \big) \Big] \dx. \\
    \end{align*}
    Further, integrating by parts leads to
    \begin{multline}
      \label{eq:energy-derivative-Calculation}
      \mathscr{E}_{\phi,g_1,g_2}'(\tau) = \frac{1}{2} \intlim{0}{\ell} \Big[ g_1 A (\partial_t u + \partial_x u)^2 - g_2 A (\partial_t u - \partial_x u)^2 \Big] \\
      + \frac{1}{2} \int_0^\ell \Big[ -\partial_x g_1 A (\partial_t u + \partial_x u)^2 + \partial_x g_2 A (\partial_t u - \partial_x u)^2 \\
      - \partial_x A (g_1 - g_2) \big((\partial_t u)^2 - (\partial_x u)^2\big) \Big] \dx.
  \end{multline}
Recall that $A(x)$ is constant for all $x>\ell$ so there are no incoming waves at $x=\ell$. This means that $(\partial_t + \partial_x) u (\ell, \tau)=0$. We also take the requirement
\begin{equation}
  \label{eq:energy-coefficient-zero-values}
  g_1(0) = g_2(0) = 1.
\end{equation}
Then the {boundary terms} satisfy
\begin{align*}
  &\intlim{0}{\ell} \Big[ g_1 A (\partial_t u + \partial_x u)^2 - g_2 A (\partial_t u - \partial_x u)^2 \Big] \\
  &= g_1(\ell) A(\ell) (\partial_t u + \partial_x u)^2(\ell,\tau) - g_2(\ell) A(\ell) (\partial_t u - \partial_x u)^2(\ell,\tau) \\
  &\quad - g_1(0) A(0) (\partial_t u + \partial_x u)^2(0,\tau) + g_2(0) A(0) (\partial_t u - \partial_x u)^2(0,\tau) \\
  &= 0 - g_2(\ell) A(\ell) (\partial_t u - \partial_x u)^2(\ell,\tau)-   4 \partial_t u(0,\tau) \partial_x u(0,\tau) \\
  &= - g_2(\ell) A(\ell) (\partial_t u - \partial_x u)^2(\ell,\tau) - 4 \partial_t  \ND ( \phi) |_{t= \tau} \phi (\tau) .
\end{align*}
The first term on the right is always non-positive. Moreover, $\phi (\tau) = 0$ for $\tau > \max \supp \phi $ which implies that the second term vanishes for a such $\tau$. Thus, the {expression for boundary terms} above is, { in total,} non-positive for $\tau > \max \supp \phi $. 
We can also estimate the integrand of the remaining integral in (\ref{eq:energy-derivative-Calculation}) as follows
\[
  - \partial_x A (g_1 - g_2) \big((\partial_t u)^2 - (\partial_x u)^2\big) \leq \tfrac{1}{2} \abs{\partial_x A} \abs{g_1 - g_2} \big((\partial_t u + \partial_x u)^2 + (\partial_t u - \partial_x u)^2\big). 
\]
These then imply 
\begin{align*}
  &\mathscr{E}_{\phi, g_1,g_2}'(\tau) 
   \leq \frac{1}{2} \int_0^\ell \parentheses*{-\partial_x g_1 + \frac{1}{2} \abs*{\frac{\partial_x A}{A}} \abs{g_1 - g_2}} (\partial_t u + \partial_x u)^2 A \dx \\
  &\quad \frac{1}{2} \int_0^\ell \parentheses*{\partial_x g_2 + \frac{1}{2} \abs*{\frac{\partial_x A}{A}} \abs{g_1 - g_2}} (\partial_t u - \partial_x u)^2 A \dx.
\end{align*}
If $g_1$ and $g_2$ additionally satisfy
\begin{equation}
  \label{eq:energy-coefficient-inequalities}
  \begin{cases}
    \phantom{-} \partial_x g_1 - \tfrac{M}{2} \abs{g_1 - g_2} \geq \lambda g_1 \\
    - \partial_x g_2 - \tfrac{M}{2} \abs{g_1 - g_2} \geq \lambda g_2 ,
  \end{cases}
\end{equation}
for  $M = \max \abs{\partial_x A / A}$,  we arrive at
\begin{equation}
  \label{eq:energy-derivative-estimate}
  \mathscr{E}_{\phi ,g_1, g_2}'(\tau) \leq - \lambda  \mathscr{E}_{\phi,g_1,g_2}(\tau),
\end{equation}
whenever $\tau > \max \supp \phi$. This will give an energy estimate thanks to Gr\"onwall's inequality. 
It is indeed possible to find such $g_1,g_2$. Following \cite{arnold22_expon_time_decay_one_dimen} in our context, we summarize:

\begin{lemma}\label{gronw-lemma}
  Let $M = \max \abs{A^{-1} \partial_x A}$ and $\lambda = \tfrac{M}{2} e^{-M \ell} \sqrt{1 - 2 M \ell e^{-2 M \ell}}$ where $\ell$ is the length of the domain of interest defined along with $\mathscr{E}_{\phi,g_1,g_2}$ in (\ref{eq:energy-def}). Set
  \begin{align}
    g_1(x) = \frac{e^{\frac{M}{2} x}}{\sqrt{\lambda^2 + \parentheses*{M/2}^2}}
    \Bigg(& \sqrt{\lambda^2 + \parentheses*{M/2}^2} \cosh \parentheses*{\sqrt{\lambda^2 + \parentheses*{M/2}^2} x} \\
          & + \parentheses*{\lambda - M/2} \sinh \parentheses*{\sqrt{\lambda^2 + \parentheses*{M/2}^2} x} \Bigg)
  \end{align}
  \begin{align}
    g_2(x) = \frac{e^{\frac{M}{2} x}}{\sqrt{\lambda^2 + \parentheses*{M/2}^2}}
    \Bigg(& \sqrt{\lambda^2 + \parentheses*{M/2}^2} \cosh \parentheses*{\sqrt{\lambda^2 + \parentheses*{M/2}^2} x} \\
          & - \parentheses*{\lambda + M/2} \sinh \parentheses*{\sqrt{\lambda^2 + \parentheses*{M/2}^2} x} \Bigg)
  \end{align}
  Then $g_1$ and $g_2$ are strictly %{ (tarkista)} 
  positive on $[0,\ell]$, and they satisfy (\ref{eq:energy-coefficient-zero-values}) and (\ref{eq:energy-coefficient-inequalities}). 
  As a consequence, (\ref{eq:energy-derivative-estimate}) holds for $\tau >  \tau_0:= \max \supp  \phi $ and hence
    \begin{equation}
    \label{eq:energy-local-decay}
    \mathscr{E}_{\phi,g_1,g_2}(\tau) \leq \mathscr{E}_{\phi,g_1,g_2}(\tau_0) e^{-\lambda (\tau-\tau_0)}
  \end{equation}
  by Gr\"onwall's inequality.

\end{lemma}

\bigskip We will need a result tying together the energy decay inside, and the decay of our measurement.
\begin{lemma}
  \label{lem:measurement-decay}
  Let $u$ satisfy 
  \begin{align}
    \square_A u &= 0, && x>0,\, t\in\R,\\
    \partial_x u (0,t) &= \phi(t), && t\in\R,\\
    u(x,t) &= 0, && x>0,\, t<0.
  \end{align}
  where $\phi \in C^{\infty}_c (\R)$ is supported in $(0,\infty)$. Assume furthermore that $A(x)=A_\infty$ for $x > \ell$. Then
  \begin{equation}
    \big|  u(0,t) -c(t)  \big|^2 \leq \int_0^{\ell} \left(   \partial_x u (x,t) + \frac{A(x)}{A_\infty} \partial_t u(x,t) \right)^2  \dx.  
  \end{equation}
  where
  \begin{equation}
    c(t) = - \frac{1}{A_\infty} \int_0^t \phi (t) \dt.
  \end{equation}
\end{lemma}
\begin{proof}
  Integrate the first equation above in the rectangle $(0,\infty) \times (0,t)$ and use $A(0) = 1, \ A(\ell) = A_\infty$ to get
  \begin{align}
    0 &= \int_0 ^{\ell} \int_0^t \big[ A(x) \partial_t^2 u(x,t) - \partial_x (A(x)  \partial_x u(x,t) ) \big] \ds \dx \\
      &= \int_0^{\ell} \big[ A(x) \partial_t u(x,t) - A(x) \partial_t u(x,0) \big] \dx - \int_0^t\big[  A_\infty \partial_x u(\ell,s) -  \partial_x u(0,s) \big] \ds.
  \end{align}
Using the fact that $ \partial_t u(x,0) = 0$ for all $x>0$, we get  
\[
 \int_0^{\ell} A(x) \partial_t u(x,t)  \dx - \int_0^t\big[  A_\infty \partial_x u(\ell,s) -  \partial_x u(0,s) \big] \ds = 0.
\]
Moreover, $(\partial_t u+ \partial_x u )(\ell,s) = 0$ holds since $A(x)=A_\infty$ on $(\ell,\infty)$ implies that any wave on that set must be moving to the right. 
Indeed, the support of $u$ would otherwise propagate along the rays $x+t =$constant to $(0,\infty) \times (-\infty,0)$ where $u=0$ by definition. 
Hence, we can replace $\partial_x u (\ell,s)$ in the second integral by $-\partial_t u (\ell,s)$: 
\[
 \int_0^{\ell}  A(x) \partial_t u(x,t)   \dx + \int_0^t\big[ A_\infty \partial_t u(\ell,s)  +  \partial_x u(0,s) \big] \ds = 0.
\]
After integrating with respect to $s$ and recalling that $u(x,0) = 0$, 
we have
  \begin{equation}
     \int_0^{\ell} A(x) \partial_t u (x,t) \dx +  A_\infty u(\ell,t) + \int_0^t  \partial_x u(0,s) \ds = 0.
  \end{equation}
  Recall that $ \partial_x u(0,s) = \phi(s)$. 
Hence, 
  \begin{align}
      u(0,t) +    \frac{1}{A_\infty}  \int_0^t \phi (s) \ds  &=   \big( u(0,t) -   u(\ell,t) \big) -   \frac{1}{A_\infty} \int_0^{\ell}A(x) \partial_t u(x,t) \dx \\
	&=   \int_0^\ell \left[  \partial_x u(x,t)  -  \frac{A(x)}{A_\infty}  \partial_t u(x,t) \right] \dx.
  \end{align}
This gives
    \begin{align}
     \left|  u(0,t) +    \frac{1}{A_\infty}  \int_0^t \phi (s) \ds  \right|  & \leq  \int_0^\ell \left|  \partial_x u(x,t)  -  \frac{A(x)}{A_\infty}  \partial_t u(x,t) \right| \dx. 
  \end{align}
from which the claim follows by Hölder's inequality. 
\end{proof}

\begin{proposition}\label{1decay-pro}
 Let $u = u_\phi  \in C^\infty( [0,\infty) \times \R)  $ satisfy 
  \begin{align}
    \square_A u(x,t) &= 0, &&   (x,t) \in (0,\infty) \times \R,\\
    \partial_x u (0,t) &= \phi(t), && t\in\R,\\
    u(x,t) &= 0, && (x,t) \in (0,\infty) \times (-\infty,0).
  \end{align}
  where $\phi \in C^{\infty}_c (\R)$ is supported in $(0,\infty)$. Assume furthermore that $A(x)=A_\infty$ for $x > \ell$. Then, there are $C,C'>0$ such that 
 \[
    \big|  u(0,t) -c(t)  \big| \leq C e^{-C't}, \quad \forall t \in \R, 
    \]
    where  $c(t) := - \frac{1}{A_\infty} \int_0^t \phi(s) ds$. 
\end{proposition}
\begin{proof}
Let $g_1,g_2$ be as in Lemma $\ref{gronw-lemma}$. 
By Lemma \ref{lem:measurement-decay},
 \[
 \big|  u(0,t) -c(t)  \big|^2 \leq \int_0^{\ell} \left(   \partial_x u (x,t) + \frac{A(x)}{A_\infty} \partial_t u(x,t) \right)^2  \dx
 \]
 Denote $X :=  \partial_x u (x,t) $ and $Y = \partial_t u(x,t) $. For $x \in (0,\ell)$, 
 \begin{align}
 \left(   \partial_x u (x,t) + \frac{A(x)}{A_\infty} \partial_t u(x,t) \right)^2 =  X^2 +2XY \frac{A(x)}{A_\infty} +   \frac{A^2(x)}{A_\infty^2}Y^2\\
  \leq  \begin{cases}   \left(1+\frac{A(x)}{A_\infty} + \frac{A^2(x)}{A_\infty^2} \right)  (X^2 +2XY +   Y^2) , & \text{if} \quad   XY \geq 0 , \\
 \left(1+\frac{A(x)}{A_\infty} + \frac{A^2(x)}{A_\infty^2} \right)  ( X^2 -2XY +  Y^2), & \text{if} \quad   XY <0  \end{cases} \\
   =   \begin{cases}  \left(1+\frac{A(x)}{A_\infty} + \frac{A^2(x)}{A_\infty^2} \right)   (X + Y)^2 , & \text{if} \quad   XY \geq 0 , \\
  \left(1+\frac{A(x)}{A_\infty} + \frac{A^2(x)}{A_\infty^2} \right)   ( X-Y)^2, & \text{if} \quad   XY <0  \end{cases} \\
 \leq K \Big(  g_1(x)  (X+Y)^2+   g_2(x) (X-Y)^2 \Big) A(x) ,
 \end{align}
 where 
 \[
 K = \frac{ \left(1+\frac{\max A}{A_\infty} + \frac{\max A^2}{A_\infty^2} \right)  }{[ \min  A ] [  \min_{x \in [0,\ell] }  ( \min \{  g_1(x) , g_2(x) \} ) ]}.
 \]
 Thus, 
 \[
  \big|  u(0,t) -c(t)  \big|^2\leq   K \int_0^{\ell}  \Big(  g_1(x)  (X+Y)^2+   g_2(x) (X-Y)^2 \Big) A(x)   \dx = K  \mathscr{E}_{\phi,g_1,g_2}(t). 
 \]
 By Lemma \ref{gronw-lemma}, $\mathscr{E}_{\phi,g_1,g_2}(t) \leq \mathscr{E}_{\phi,g_1,g_2}(\tau_0) e^{-\lambda (t-\tau_0)}$. 
 In conclusion, 
 \[
  \big|  u(0,t) -c(t)  \big|^2\leq  C e^{-C't},
 \]
for $C= K  \mathscr{E}_{\phi,g_1,g_2}(\tau_0)  e^{\lambda \tau_0}$ and $C' = \lambda$. 
\end{proof}

%\appendix
%\appendixpage

\section{Sondhi and Gopinath 1971}\label{sondhi_app}

We briefly describe  how the results in \cite{sondhi_gopinath} can be applied to determine $A$ from the Neumann-to-Dirichlet data. 
Given $f  \in  \mathcal{E}'(\R)  $, let $w= w_f \in C^\infty( (0,\infty) \times \R) $ satisfy the equations 
\begin{align}
 & \square_A w (x,t)  = 0,  && \text{for} \quad (x,t) \in    (0,\infty) \times \R ,\label{Atr:eq:wave-u}\\
 & \partial_x w |_{x=0} (t)  =  f(t), &&\text{for} \quad  t \in \R,\label{Atr:eq:neumann-u}\\
 & w (x,t)= 0, &&\text{for} \quad (x,t) \in (0,\infty) \times (- \infty, t_0 ) .\label{Atr:eq:initial-u}
\end{align} 
where $t_0 <  \min \text{supp} (f)$.
Then, $\Lambda_{ND} ( f ) = w |_{x=0}$. 
The wave equation $\square_A w (x,t)  = 0$ can be rewritten as a system
\begin{align}
&A \partial_x p = - \partial_t u \label{kkokk12} \\
&A \partial_t p = - \partial_x u, \label{kokk122}
\end{align}
where $p = -\partial_t w$ and $u =  A \partial_x w$.
As $A(0) =1$, we have $p (0, \tau)  = - \partial_t \ND( f)$. Consider symmetric $\phi \in C^\infty_c( -a, a)$, $a>0$.  
According to \cite[eq. (5)]{sondhi_gopinath}
\[
 - \partial_t \ND[ \phi]  (t) = \int_{-a }^t \phi(\tau) \hat{h}(t-\tau) d\tau, \quad \forall t \in (-\infty ,a)
\]
where $\hat{h} = \delta_0 + h$ is the impulse response, as in the reference. 
For a sequence $\phi_n \in C^\infty_c  (-a,a) $ converging  to $\delta_0$ in $\mathcal{E}'(\R)$, we get 
\[
 - \partial_t \ND[ \phi_n]  (t)  \to   h(t), \quad \forall t \in (-\infty,a),
\]
as $n \to \infty$ by continuity. Since $a>0$ is arbitrary, we conclude that $\ND$ uniquely determines $h$ on $\R$. That is;
\begin{lemma}\label{sondhi_lemma}
Let $\Lambda_\nu$, $\nu =1,2$ be the Neumann-to-Dirichlet maps related to admissible $A_\nu$, $\nu =1,2$. 
If $\Lambda_{1} = \Lambda_{2}$  holds on $\mathcal{E}'(\R)$ for the Neumann-to-Dirichlet maps of two admissible $A_1,A_2$, the associated impulse response functions $h_1$ and $h_2$ are identical. 
\end{lemma}

Let $a>0$. 
The equation \eqref{kokk122} yields 
\[
  A(x)  p (x,t_0 + a)  = \int_{t_0}^{t_0 + a}  A(x) \partial_t p (x,t) dt =    -   \int_{t_0}^{t_0 + a}  \partial_x u(x,t) dt ,
\]
where we used the initial value  $p(x,t_0) = 0$. 
The integral is well defined provided enough regularity in $f$. 
Integrating again in $x$ yields 
\begin{align}
\int_0^a  A(x)  p (x,t_0 + a) dx  =   -   \int_{t_0}^{t_0 + a}   u(a,t) dt   +  \int_{t_0}^{t_0 + a}  u(0,t) dt .\label{ad00}
\end{align}
Since the wave $w$ satisfies $w|_{t = t_0} = 0 = \partial_t w|_{t = t_0}$, it vanishes in the wedge $\{ (x,t) \in (0,\infty) \times \R : t -t_0 < x \}$.      
Consequently, the first term on the right in \eqref{ad00} vanishes and we deduce 
\begin{align}
\int_0^a  A(x)  p (x,t_0 + a) dx  =    \int_{t_0}^{t_0 + a}  u(0,t) dt .\label{1ad00}
\end{align}
This is the equation (3) in  \cite{sondhi_gopinath}. 
As $A(0) = 1$, it coincides with 
\begin{align}
\int_0^a  A(x)  p (x,t_0 + a) dx  =    \int_{t_0}^{t_0 + a}   f(t)  dt ,\label{1ad001}.
\end{align}
We can now show that not only $h_1 = h_2$ but also $A_1 = A_2$ holds provided that $\Lambda_{1} = \Lambda_{2}$. 
\begin{theorem}\label{sondhi_theorem}
If $\Lambda_{1} = \Lambda_{2}$ holds on $\mathcal{E}'(\R)$ for the Neumann-to-Dirichlet maps of two admissible $A_1,A_2$, then $A_1 = A_2$. 
\end{theorem}
\begin{proof}
Let the index $\nu = 1,2$ indicate the choice of $A = A_\nu$.  
It was shown in the reference that there is a convenient compactly supported Neumann boundary value $f_{\nu}=f_{a,\nu}$  such that 
\begin{align}
p_\nu (x,t_0 +a) =1, \quad  \text{for}  \quad  x \in [0,a]. \label{1i2}
\end{align} 
Moreover, such an input $f_{a,\nu}$ obeys $(1+H_{a, \nu}) f_{a,\nu} =  1$ where $(1+H_{a,\nu})$ is an invertible operator (see \cite[Appendix A]{sondhi_gopinath} for more details) that depends solely on the impulse response. 
Thus, if $\Lambda_{1} = \Lambda_{2}$ for the Neumann-to-Dirichlet maps of two admissible $A_1,A_2$, then the associated operators 
$(1+H_{a,\nu}) $, $\nu = 1,2$ coincide by the previous lemma. By the injectivity of the operators, we have $f_{a,1}= f_{a,2}$.  
Applying \eqref{1i2} to  \eqref{1ad001} yields
\[
\int_0^a  A_1(x) dx   =  \int_{t_0}^{t_0 + a}   f_{a,1}(t)  dt =  \int_{t_0}^{t_0 + a}   f_{a,2}(t)  dt  = \int_0^a  A_2(x) dx
\]
 from which $A_1 (a) = A_2(a)$ then follows by taking a derivative. Since $a>0$ was arbitrary, we conclude $A_1 = A_2$.
\end{proof}

\printbibliography

@article{knox2020determining,
  title={Determining both the source of a wave and its speed in a medium from boundary measurements},
  author={Knox, Christina and Moradifam, Amir},
  journal={Inverse Problems},
  volume={36},
  number={2},
  pages={025002},
  year={2020},
  publisher={IOP Publishing}
}

@article{liu2015determining,
  title={Determining both sound speed and internal source in thermo-and photo-acoustic tomography},
  author={Liu, Hongyu and Uhlmann, Gunther},
  journal={Inverse Problems},
  volume={31},
  number={10},
  pages={105005},
  year={2015},
  publisher={IOP Publishing}
}

@article{kian2023determination,
  title={Determination of the sound speed and an initial source in photoacoustic tomography},
  author={Kian, Yavar and Uhlmann, Gunther},
  journal={arXiv preprint arXiv:2302.03457},
  year={2023}
}

@article{avdonin2014reconstructing,
  title={Reconstructing the potential for the one-dimensional Schr{\"o}dinger equation from boundary measurements},
  author={Avdonin, Sergei A and Mikhaylov, Victor S and Ramdani, Karim},
  journal={IMA Journal of Mathematical Control and Information},
  volume={31},
  number={1},
  pages={137--150},
  year={2014},
  publisher={Oxford University Press}
}

@article{finch2013transmission,
  title={Transmission eigenvalues and thermoacoustic tomography},
  author={Finch, David and Hickmann, Kyle S},
  journal={Inverse Problems},
  volume={29},
  number={10},
  pages={104016},
  year={2013},
  publisher={IOP Publishing}
}

@article{jing2021simultaneous,
  title={Simultaneous uniqueness for multiple parameters identification in a fractional diffusion-wave equation},
  author={Jing, Xiaohua and Yamamoto, Masahiro},
  journal={arXiv preprint arXiv:2103.07720},
  year={2021}
}

@inproceedings{bukhgeim1981global,
  title={Global uniqueness of a class of multidimensional inverse problems},
  author={Bukhgeim, Aleksandr L'vovich and Klibanov, Mikhail Viktorovich},
  booktitle={Doklady Akademii Nauk},
  volume={260},
  number={2},
  pages={269--272},
  year={1981},
  organization={Russian Academy of Sciences}
}

@article{cheng2002identification,
  title={Identification of convection term in a parabolic equation with a single measurement},
  author={Cheng, Jin and Yamamoto, Masahiro},
  journal={Nonlinear Analysis: Theory, Methods \& Applications},
  volume={50},
  number={2},
  pages={163--171},
  year={2002},
  publisher={Elsevier}
}

@article{cheng2009uniqueness,
  title={Uniqueness in an inverse problem for a one-dimensional fractional diffusion equation},
  author={Cheng, Jin and Nakagawa, Junichi and Yamamoto, Masahiro and Yamazaki, Tomohiro},
  journal={Inverse problems},
  volume={25},
  number={11},
  pages={115002},
  year={2009},
  publisher={IOP Publishing}
}

@article{feizmohammadi2023global,
  title={Global recovery of a time-dependent coefficient for the wave equation from a single measurement},
  author={Feizmohammadi, Ali and Kian, Yavar},
  journal={Asymptotic Analysis},
  volume={131},
  number={3-4},
  pages={513--539},
  year={2023},
  publisher={SAGE Publications Sage UK: London, England}
}

@article{helin2014inverse,
  title={Inverse problem for the wave equation with a white noise source},
  author={Helin, Tapio and Lassas, Matti and Oksanen, Lauri},
  journal={Communications in Mathematical Physics},
  volume={332},
  pages={933--953},
  year={2014},
  publisher={Springer}
}

@article{helin2012inverse,
  title={An inverse problem for the wave equation with one measurement and the pseudorandom source},
  author={Helin, Tapio and Lassas, Matti and Oksanen, Lauri},
  journal={Analysis \& PDE},
  volume={5},
  number={5},
  pages={887--912},
  year={2012},
  publisher={Mathematical Sciences Publishers}
}

@article{Rakesh_2001,
doi = {10.1088/0266-5611/17/5/311},
url = {https://dx.doi.org/10.1088/0266-5611/17/5/311},
year = {2001},
publisher = {},
volume = {17},
number = {5},
pages = {1401},
author = {Rakesh},
title = {A one-dimensional inverse problem
for a hyperbolic system with complex  coefficients},
journal = {Inverse Problems}
}

@article{stefanov2013recovery,
  title={Recovery of a source term or a speed with one measurement and applications},
  author={Stefanov, Plamen and Uhlmann, Gunther},
  journal={Transactions of the American Mathematical Society},
  volume={365},
  number={11},
  pages={5737--5758},
  year={2013}
}

@article{helin2020inverse,
  title={Inverse problems for heat equation and space--time fractional diffusion equation with one measurement},
  author={Helin, Tapio and Lassas, Matti and Ylinen, Lauri and Zhang, Zhidong},
  journal={Journal of Differential Equations},
  volume={269},
  number={9},
  pages={7498--7528},
  year={2020},
  publisher={Elsevier}
}

@article{helin2018atmospheric,
  title={Atmospheric turbulence profiling with unknown power spectral density},
  author={Helin, Tapio and Kindermann, Stefan and Lehtonen, Jonatan and Ramlau, Ronny},
  journal={Inverse Problems},
  volume={34},
  number={4},
  pages={044002},
  year={2018},
  publisher={IOP Publishing}
}

@article{caro2019inverse,
  title={Inverse scattering for a random potential},
  author={Caro, Pedro and Helin, Tapio and Lassas, Matti},
  journal={Analysis and Applications},
  volume={17},
  number={04},
  pages={513--567},
  year={2019},
  publisher={World Scientific}
}

@article{garnier2012correlation,
  title={Correlation-based virtual source imaging in strongly scattering random media},
  author={Garnier, Josselin and Papanicolaou, George},
  journal={Inverse Problems},
  volume={28},
  number={7},
  pages={075002},
  year={2012},
  publisher={IOP Publishing}
}

@article{garnier2009passive,
  title={Passive sensor imaging using cross correlations of noisy signals in a scattering medium},
  author={Garnier, Josselin and Papanicolaou, George},
  journal={SIAM Journal on Imaging Sciences},
  volume={2},
  number={2},
  pages={396--437},
  year={2009},
  publisher={SIAM}
}

@article{devaney1979inverse,
  title={The inverse problem for random sources},
  author={Devaney, AJ},
  journal={Journal of Mathematical Physics},
  volume={20},
  number={8},
  pages={1687--1691},
  year={1979},
  publisher={American Institute of Physics}
}

@article{triki2024fourier,
  title={Fourier method for inverse source problem using correlation of passive measurements},
  author={Triki, Faouzi and Linder-Steinlein, Kristoffer and Karamehmedovi{\'c}, Mirza},
  journal={Inverse Problems},
  volume={40},
  number={10},
  pages={105009},
  year={2024},
  publisher={IOP Publishing}
}

@article{agaltsov2020global,
  title={Global uniqueness in a passive inverse problem of helioseismology},
  author={Agaltsov, Alexey and Hohage, T and Novikov, RG},
  journal={Inverse Problems},
  volume={36},
  number={5},
  year={2020},
  publisher={IOP Pub.}
}

@article{li2024stability,
  title={Stability estimates of inverse random source problems for the wave equations by using correlation-based data},
  author={Li, Peijun and Liang, Ying and Wang, Xu},
  journal={arXiv preprint arXiv:2410.07938},
  year={2024}
}

@article{li2020inverse,
  title={Inverse random source problems for time-harmonic acoustic and elastic waves},
  author={Li, Jianliang and Helin, Tapio and Li, Peijun},
  journal={Communications in Partial Differential Equations},
  volume={45},
  number={10},
  pages={1335--1380},
  year={2020},
  publisher={Taylor \& Francis}
}

@article{li2022inverse1,
  title={Inverse source problems for the stochastic wave equations: far-field patterns},
  author={Li, Jianliang and Li, Peijun and Wang, Xu},
  journal={SIAM Journal on Applied Mathematics},
  volume={82},
  number={4},
  pages={1113--1134},
  year={2022},
  publisher={SIAM}
}

@article{bao2017inverse,
  title={Inverse random source scattering for elastic waves},
  author={Bao, Gang and Chen, Chuchu and Li, Peijun},
  journal={SIAM Journal on Numerical Analysis},
  volume={55},
  number={6},
  pages={2616--2643},
  year={2017},
  publisher={SIAM}
}

@article{li2022inverse2,
  title={An inverse random source problem for the biharmonic wave equation},
  author={Li, Peijun and Wang, Xu},
  journal={SIAM/ASA Journal on Uncertainty Quantification},
  volume={10},
  number={3},
  pages={949--974},
  year={2022},
  publisher={SIAM}
}

@Misc{ali_passive,
 Author = {Ali Feizmohammadi},
 Title = {{Reconstruction of 1-D evolution equations and their initial data from
  one passive measurement}},
 Year = {2025},
 eprint = {2501.10590}
}

@Book{GP_passive,
  author    = {Garnier, Josselin and Papanicolaou, George},
  title     = {Passive imaging with ambient noise},
  year      = {2016},
  publisher = {Cambridge: Cambridge University Press},
  language  = {English},
  doi       = {10.1017/CBO9781316471807},
  zbmath    = {6551425},
  zbl       = {1352.86001}
}

@article{alku11_glott_inver_filter_analy_human,
  author       = {Paavo Alku},
  title        = {Glottal Inverse Filtering Analysis of Human Voice
                  Production -- A Review of Estimation and
                  Parameterization Methods of the Glottal Excitation
                  and Their Applications},
  journal      = {Sadhana},
  volume       = {36},
  number       = {5},
  pages        = {623-650},
  year         = {2011},
  doi          = {10.1007/s12046-011-0041-5},
  url          = {https://doi.org/10.1007/s12046-011-0041-5},
  DATE_ADDED   = {Fri Sep 3 22:40:57 2021},
}

@article{arnold22_expon_time_decay_one_dimen,
  author       = {Anton Arnold and Sjoerd Geevers and Ilaria Perugia
                  and Dmitry Ponomarev},
  title        = {On the Exponential Time-Decay for the
                  One-Dimensional Wave Equation With Variable
                  Coefficients},
  journal      = {Communications on Pure and Applied Analysis},
  volume       = {21},
  number       = {10},
  pages        = {3389},
  year         = {2022},
  doi          = {10.3934/cpaa.2022105},
  url          = {http://dx.doi.org/10.3934/cpaa.2022105},
  DATE_ADDED   = {Tue Sep 5 17:04:43 2023},
}

@article{bao16_inver_random_sourc_scatt_probl_sever_dimen,
  author       = {Gang Bao and Chuchu Chen and Peijun Li},
  title        = {Inverse Random Source Scattering Problems in Several
                  Dimensions},
  journal      = {SIAM/ASA Journal on Uncertainty Quantification},
  volume       = {4},
  number       = {1},
  pages        = {1263--1287},
  year         = {2016},
  doi          = {10.1137/16m1067470},
  url          = {https://doi.org/10.1137/16m1067470},
  DATE_ADDED   = {Tue Mar 7 03:28:56 2023},
}

@article{blåsten14_corner_alway_scatt,
  author       = {Bl{\r a}sten, E. and P\"aiv\"arinta, Lassi and
                  Sylvester, J.},
  title        = {Corners Always Scatter},
  journal      = {Communications in Mathematical Physics},
  volume       = {331},
  number       = {2},
  pages        = {725-753},
  year         = {2014},
  doi          = {10.1007/s00220-014-2030-0},
  url          = {https://doi.org/10.1007/s00220-014-2030-0},
  issn         = {1432-0916},
  month        = {4},
  publisher    = {Springer Science and Business Media LLC},
}

@article{blåsten20_unique_deter_shape_scatt_screen,
  author       = {Bl{\r a}sten, Emilia and P\"aiv\"arinta, Lassi and
                  Sadique, Sadia},
  title        = {Unique Determination of the Shape of a Scattering
                  Screen From a Passive Measurement},
  journal      = {Mathematics},
  volume       = {8},
  number       = {7},
  pages        = {1156},
  year         = {2020},
  doi          = {10.3390/math8071156},
  url          = {https://doi.org/10.3390/math8071156},
  issn         = {2227--7390},
  month        = {7},
  publisher    = {MDPI AG},
}

@article{buckingham92_imagin_ocean_with_ambien_noise,
  author       = {Michael J. Buckingham and Broderick V. Berknout and
                  Stewart A. L. Glegg},
  title        = {Imaging the Ocean With Ambient Noise},
  journal      = {Nature},
  volume       = {356},
  number       = {6367},
  pages        = {327--329},
  year         = {1992},
  doi          = {10.1038/356327a0},
  url          = {http://dx.doi.org/10.1038/356327a0},
}

@book{duistermaat-FIOs,
  author       = {Duistermaat, J J},
  title        = {Fourier Integral Operators},
  year         = {1996},
  publisher    = {Birkh\"auser Boston},
  doi          = {10.1007/978-0-8176-8108-1},
  url          = {https://doi.org/10.1007/978-0-8176-8108-1},
}

@book{friedlander-joshi-book,
  author       = {Friedlander, F G and Joshi, Mark},
  title        = {Introduction to the Theory of Distributions},
  year         = {1999},
  publisher    = {Cambridge University Press},
  edition      = {2},
  isbn         = {0521649714},
}

@misc{gelfand55_deter_differ_equat_from_its_spect_funct,
  author       = {Gel'fand, I. M. and Levitan, B. M.},
  title        = {On the Determination of a Differential Equation From
                  Its Spectral Function},
  journal      = {Eleven Papers on Topology, Function Theory and
                  Differential Equations},
  year         = {1955},
  doi          = {10.1090/trans2/001/11},
  url          = {http://dx.doi.org/10.1090/trans2/001/11},
  pages        = {253--304},
  issn         = {2472-3193},
  publisher    = {American Mathematical Society},
}

@article{gozum22_noise_based_high_resol_time,
  author       = {Gozum, Murat M. and Nasraoui, Saber and
                  Grigoropoulos, Georgios and Louati, Moez and
                  Ghidaoui, Mohamed S.},
  title        = {A Noise-Based High-Resolution Time-Reversal Method
                  for Acoustic Defect Localization in Water Pipes},
  journal      = {The Journal of the Acoustical Society of America},
  volume       = {152},
  number       = {6},
  pages        = {3373--3383},
  year         = {2022},
  doi          = {10.1121/10.0016502},
  url          = {http://dx.doi.org/10.1121/10.0016502},
  issn         = {1520-8524},
  month        = 12,
  publisher    = {Acoustical Society of America (ASA)},
}

@book{grigis-sjostrand,
  author       = {Alain Grigis and Johannes Sj\"ostrand},
  title        = {Microlocal Analysis for Differential Operators: An
                  Introduction},
  year         = {1994},
  publisher    = { Cambridge University Press},
  doi          = {10.1017/CBO9780511721441},
  url          = {https://doi.org/10.1017/CBO9780511721441},
  isbn         = {9780511721441},
  series       = { London Mathematical Society Lecture Note Series
                  (196)},
}

@article{helin18_correl_based_passiv_imagin_b,
  author       = {T. Helin and M. Lassas and L. Oksanen and
                  T. Saksala},
  title        = {Correlation Based Passive Imaging With a White Noise
                  Source},
  journal      = {Journal de Mathématiques Pures et Appliquées},
  volume       = {116},
  number       = {nil},
  pages        = {132--160},
  year         = {2018},
  doi          = {10.1016/j.matpur.2018.05.001},
  url          = {http://dx.doi.org/10.1016/j.matpur.2018.05.001},
  DATE_ADDED   = {Wed Oct 26 21:13:26 2022},
}

@book{hida-lectures-book,
  author       = {Takeyuki Hida and Si Si},
  title        = {Lectures on White Noise Functionals},
  year         = {2008},
  publisher    = {World Scientific Publishing Co. Pte. Ltd.},
  doi          = {10.1142/5664},
  url          = {https://doi.org/10.1142/5664},
  isbn         = {978-981-256-052-0},
}

@article{hu16_shape_ident_inver_medium_scatt,
  author       = {Guanghui Hu and Mikko Salo and Esa V. Vesalainen},
  title        = {Shape Identification in Inverse Medium Scattering
                  Problems With a Single Far-Field Pattern},
  journal      = {SIAM Journal on Mathematical Analysis},
  volume       = {48},
  number       = {1},
  pages        = {152--165},
  year         = {2016},
  doi          = {10.1137/15m1032958},
  url          = {https://doi.org/10.1137/15m1032958},
  DATE_ADDED   = {Sat Sep 4 00:10:14 2021},
}

@article{isserlis18_formul_produc_momen_coeff_any,
  author       = {L. Isserlis},
  title        = {On a Formula for the Product-Moment Coefficient of
                  Any Order of a Normal Frequency Distribution in Any
                  Number of Variables},
  journal      = {Biometrika},
  volume       = {12},
  number       = {1/2},
  pages        = {134},
  year         = {1918},
  doi          = {10.2307/2331932},
  url          = {http://dx.doi.org/10.2307/2331932},
  DATE_ADDED   = {Tue Jun 20 19:19:37 2023},
}

@book{kuo-book,
  author       = {Hui-Hsiung Kuo},
  title        = {White Noise Distribution Theory},
  year         = {1996},
  publisher    = {CRC Press},
  doi          = {10.1201/9780203733813},
  url          = {https://doi.org/10.1201/9780203733813},
  isbn         = {9780138733810},
}

@book{lax89_scatt_theor,
  author       = {Peter Lax and Ralph Phillips},
  title        = {Scattering Theory},
  year         = {1989},
  publisher    = {Academic Press},
  address      = {Boston},
  isbn         = {978-0-12-440051-1},
}

@article{liu23_inver_probl_random_schro_equat,
  author       = {Liu, Hongyu and Ma, Shiqi},
  title        = {Inverse Problem for a Random {Schr\"odinger}
                  Equation With Unknown Source and Potential},
  journal      = {Mathematische Zeitschrift},
  volume       = {304},
  number       = {2},
  pages        = {28},
  year         = {2023},
  doi          = {10.1007/s00209-023-03289-4},
  url          = {http://dx.doi.org/10.1007/s00209-023-03289-4},
  issn         = {1432-1823},
  month        = 5,
  publisher    = {Springer Science and Business Media LLC},
}

@article{sondhi_gopinath,
  author       = {M. M. Sondhi and B. Gopinath},
  title        = {Determination of Vocal-tract Shape From Impulse
                  Response At the Lips},
  journal      = {The Journal of the Acoustical Society of America},
  volume       = {49},
  number       = {6B},
  pages        = {1867--1873},
  year         = {1971},
  doi          = {10.1121/1.1912593},
  url          = {https://doi.org/10.1121/1.1912593},
  DATE_ADDED   = {Sat Sep 4 00:41:41 2021},
}

@book{titze00_princ_voice_produc,
  author       = {Titze, Ingo},
  title        = {Principles of Voice Production},
  year         = {2000},
  publisher    = {National Center for Voice and Speech},
  address      = {Iowa City, Ia},
  isbn         = {0874141222},
}

@article{zouari19_inter_pipe_area_recon_as,
  author       = {Zouari, Fedi and Blåsten, Emilia and Louati, Moez
                  and Ghidaoui, Mohamed Salah},
  title        = {Internal Pipe Area Reconstruction As a Tool for
                  Blockage Detection},
  journal      = {Journal of Hydraulic Engineering},
  volume       = {145},
  number       = {6},
  pages        = {04019019},
  year         = {2019},
  doi          = {10.1061/(asce)hy.1943-7900.0001602},
  url          = {https://doi.org/10.1061/(asce)hy.1943-7900.0001602},
  issn         = {1943-7900},
  month        = {6},
  publisher    = {American Society of Civil Engineers (ASCE)},
}

@article{gabor-wf,
    author = {Gabor, Akiva},
    title = {Remarks on the Wave Front of a Distribution},
    journal = {Transactions of the American Mathematical Society} ,
    year = {1972},
    pages = {239-244},
    doi = {10.1090/S0002-9947-1972-0341076-7}
}

@article{hormanderFIO1, 
    author ={H\"ormander, Lars},
    title= {Fourier Integral Operators. I},
    journal = {Acta Mathematica},
    year={1971},
    volume={127},
    pages={79–183},
    doi={10.1007/BF02392052}
}

@article{smirnov20_convex_1d_hyper_coeff_inver,
  author       = {Alexey Smirnov and Michael Klibanov and Loc Nguyen},
  title        = {Convexification for a {1D} Hyperbolic Coefficient
                  Inverse Problem With Single Measurement Data},
  journal      = {Inverse Problems and Imaging},
  volume       = 14,
  number       = 5,
  pages        = {913--938},
  year         = 2020,
  doi          = {10.3934/ipi.2020042},
  url          = {http://dx.doi.org/10.3934/ipi.2020042},
  DATE_ADDED   = {Thu Apr 24 16:40:37 2025},
}

@article{smirnov20_convex_inver_probl_1d_wave,
  author       = {A V Smirnov and M V Klibanov and A J Sullivan and L H Nguyen},
  title        = {Convexification for an Inverse Problem for a {1D}
                  Wave Equation With Experimental Data},
  journal      = {Inverse Problems},
  volume       = 36,
  number       = 9,
  pages        = 095008,
  year         = 2020,
  doi          = {10.1088/1361-6420/abac9a},
  url          = {http://dx.doi.org/10.1088/1361-6420/abac9a},
  DATE_ADDED   = {Thu Apr 24 16:41:57 2025},
}

@article{le22_carlem_contr_mappin_1d_inver,
  author       = {Thuy T Le and Michael V Klibanov and Loc H Nguyen
                  and Anders Sullivan and Lam Nguyen},
  title        = {Carleman Contraction Mapping for a {1D} Inverse
                  Scattering Problem With Experimental Time-Dependent
                  Data},
  journal      = {Inverse Problems},
  volume       = 38,
  number       = 4,
  pages        = 045002,
  year         = 2022,
  doi          = {10.1088/1361-6420/ac50b8},
  url          = {http://dx.doi.org/10.1088/1361-6420/ac50b8},
  DATE_ADDED   = {Thu Apr 24 16:42:39 2025},
}

@Article{romanov20_recon_princ_coeff_damped_wave,
  author       = {Vladimir Romanov and Alemdar Hasanov},
  title        = {Reconstruction of the Principal Coefficient in the
                  Damped Wave Equation From Dirichlet-To-Neumann
                  Operator},
  journal      = {Inverse Problems},
  volume       = 36,
  number       = 2,
  pages        = 025003,
  year         = 2020,
  doi          = {10.1088/1361-6420/ab53f3},
  url          = {http://dx.doi.org/10.1088/1361-6420/ab53f3},
  issn         = {1361-6420},
  month        = 1,
  publisher    = {IOP Publishing},
}

@Article{ghidaoui05_review_water_hammer_theor_pract,
  author       = {Ghidaoui, Mohamed S. and Zhao, Ming and McInnis,
                  Duncan A. and Axworthy, David H.},
  title        = {A Review of Water Hammer Theory and Practice},
  journal      = {Applied Mechanics Reviews},
  volume       = 58,
  number       = 1,
  pages        = {49–76},
  year         = 2005,
  doi          = {10.1115/1.1828050},
  url          = {http://dx.doi.org/10.1115/1.1828050},
  issn         = {2379-0407},
  month        = jan,
  publisher    = {ASME International},
}

@Book{wylie93_fluid_trans_system,
  author       = {Wylie, E Benjamin},
  title        = {Fluid Transients in Systems},
  year         = 1993,
  publisher    = {Prentice-Hall},
  address      = {London, England},
  month        = 1,
}

@Article{romanov20_recov_poten_damped_wave_equat,
  author       = {Vladimir Romanov and Alemdar Hasanov},
  title        = {Recovering a Potential in Damped Wave Equation From
                  Neumann-To-Dirichlet Operator},
  journal      = {Inverse Problems},
  volume       = 36,
  number       = 11,
  pages        = 115011,
  year         = 2020,
  doi          = {10.1088/1361-6420/abb8e8},
  url          = {http://dx.doi.org/10.1088/1361-6420/abb8e8},
  issn         = {1361-6420},
  month        = 10,
  publisher    = {IOP Publishing},
}

@Article{korpela16_regul_strat_inver_probl_1,
  author       = {Korpela, Jussi and Lassas, Matti and Oksanen, Lauri},
  title        = {Regularization Strategy for an Inverse Problem for a
                  1 + 1 Dimensional Wave Equation},
  journal      = {Inverse Problems},
  volume       = 32,
  number       = 6,
  pages        = 065001,
  year         = 2016,
  doi          = {10.1088/0266-5611/32/6/065001},
  url          = {http://dx.doi.org/10.1088/0266-5611/32/6/065001},
  issn         = {1361-6420},
  month        = 4,
  publisher    = {IOP Publishing},
}

\end{document}